\documentclass[11pt]{amsart}

\usepackage{epigamath}

\usepackage[english]{babel}

\numberwithin{equation}{section}

\usepackage{enumitem}
\usepackage{tikz-cd}

\newtheorem{theorem}{Theorem}[section]
\newtheorem{corollary}[theorem]{Corollary}
\newtheorem{lemma}[theorem]{Lemma}

\theoremstyle{definition}

\newtheorem{example}[theorem]{Example}
\newtheorem{remark}[theorem]{Remark}

\newcommand{\Z}{\mathbb Z}

\newcommand{\C}{\mathbb C}
\newcommand{\PP}{\mathbb P}
\newcommand{\R}{\mathbb R}
\newcommand{\N}{\mathbb N}

\def\E{{\mathcal E}}
\def\dbar{{\bar\partial}}

\def\PM{{\mathcal{PM}}}

\def\Ok{{\mathcal O}}

\def\1{\textbf{1}}
\DeclareMathOperator{\End}{End}
\DeclareMathOperator{\Hom}{Hom}
\DeclareMathOperator{\Id}{Id}
\DeclareMathOperator{\codim}{codim}
\DeclareMathOperator{\supp}{supp}
\DeclareMathOperator{\im}{im}
\DeclareMathOperator{\tr}{tr}
\DeclareMathOperator{\ch}{ch}
\newcommand{\Res}{\mathrm{Res}}

\EpigaVolumeYear{6}{2022} \EpigaArticleNr{14}
\ReceivedOn{November 3, 2021}
\InFinalFormOn{January 19, 2022}
\AcceptedOn{March 16, 2022}

\title{Chern currents of coherent sheaves}
\titlemark{Chern currents of coherent sheaves}

\author{Richard L\"ark\"ang}
\address{Department of Mathematics, 
  Chalmers University of Technology and
  the University of Gothenburg,
  S-412 96 Gothenburg,
  Sweden}
\email{larkang@chalmers.se}
\author{Elizabeth Wulcan}
\address{Department of Mathematics, 
  Chalmers University of Technology and
  the University of Gothenburg,
  S-412 96 Gothenburg,
  Sweden}
\email{wulcan@chalmers.se}

\authormark{R. L\"ark\"ang and E. Wulcan}

\AbstractInEnglish{%
  Given a finite locally free resolution of a
  coherent analytic sheaf $\mathcal{F}$, equipped with Hermitian metrics
  and connections, we construct an explicit current, obtained as the
  limit of certain smooth Chern forms of $\mathcal{F}$, that represents
  the Chern class of $\mathcal{F}$ and has support on the support of
  $\mathcal{F}$.  If the connections are $(1, 0)$-connections and
  $\mathcal{F}$ has pure dimension, then the first nontrivial component
  of this Chern current coincides with (a constant times) the
  fundamental cycle of $\mathcal{F}$. The proof of this goes through a
  generalized Poincar\'e--Lelong formula, previously obtained by the
  authors, and a result that relates the Chern current to the residue
  current associated with the locally free resolution.
}



\acknowledgement{%
  The first author was partially supported by the Swedish Research
  Council (2017-04908) and the second author was partially supported
  by the Swedish Research Council (2017-03905) and the Göran
  Gustafsson Foundation for Research in Natural Sciences and Medicine.
}

\begin{document}



\maketitle

\begin{prelims}

\DisplayAbstractInEnglish

\bigskip

\DisplayKeyWords

\medskip

\DisplayMSCclass

\end{prelims}


\newpage

\setcounter{tocdepth}{1}

\tableofcontents


\section{Introduction}

Let $X$ be a complex manifold and let $\mathcal F$ be a coherent analytic
sheaf on $X$ of positive codimension, \textit{i.e.} such that the support
$\supp \mathcal F$ has positive codimension.
Assume that $\mathcal F$ has a locally free resolution of the form
\begin{equation} \label{eq:Ecomplex}
  0 \to \Ok(E_N)
  \xrightarrow[]{\varphi_N}\cdots \xrightarrow[]{\varphi_1}
  \Ok(E_0) \to \mathcal{F} \to 0,
\end{equation}
where $E_k$ are holomorphic vector bundles on $X$, and $\Ok(E_k)$
denote the corresponding locally free sheaves.  Then the (total) Chern
class of $\mathcal{F}$ equals
\begin{equation} \label{eq:cClassSheafDef}
  c(\mathcal{F}) = \prod_{k = 0}^N c(E_k)^{(-1)^k},
\end{equation}
where $c(E_k)$ is the (total) Chern class of $E_k$, see
Section~\ref{kalla}.  In this paper we construct explicit
representatives of the nontrivial part of $c(\mathcal F)$ with support
on $\supp \mathcal F$.

Let us briefly describe our construction; the representatives of
$c(\mathcal F)$ will be currents obtained as limits of certain Chern
forms.
Assume that $(E, \varphi)$ is a locally free resolution of $\mathcal
F$ of the form \eqref{eq:Ecomplex}. Moreover assume that each vector
bundle $E_k$ is equipped with a Hermitian metric and a connection $D_k$
(that is not necessarily the Chern connection of the metric).
Let $\sigma_k$ be the minimal inverse of $\varphi_k$, see Section \ref{ssect:AW1}, let $\chi:\R_{\geq 0}\to \R_{\geq 0}$ be a smooth cut-off function such that
    $\chi(t) \equiv 0$ for $t \ll 1$ and $\chi(t) \equiv 1$ for $t \gg
    1$, let $s$ be a (generically non-vanishing) holomorphic section of a vector bundle such that
    $\{s = 0\}\supseteq \supp \mathcal F$, let $\chi_\epsilon = \chi(|s|^2/\epsilon)$, and let $\widehat D^\epsilon_k$ be the
connection
\begin{equation}\label{nordenstam}
\widehat D_k^\epsilon = -\chi_\epsilon \sigma_k D\varphi_k + D_k;
\end{equation}
here $D$ is a connection on $\End E$, where $E = \bigoplus E_k$, induced by the $D_k$, see Section
~\ref{smaragd}.
Then clearly the Chern form
\begin{equation}\label{lagland}
  c(E, \widehat D^\epsilon) := \prod_{k = 0}^N c(E_k, \widehat D^\epsilon_k)^{(-1)^k}
\end{equation}
is a representative for $c(\mathcal F)$. Throughout this paper we
consider Chern classes as de Rham cohomology classes of smooth forms
or currents.
Our first main result asserts that the limit of this form is a current
 with the desired properties. 

\begin{theorem} \label{thm:existence}
   Assume that $\mathcal{F}$ is a coherent analytic sheaf of positive
   codimension that admits a locally free resolution of
     the form \eqref{eq:Ecomplex}.
    Moreover, assume that each $E_k$ is equipped with a Hermitian
    metric and a connection $D_k$. Let $c(E, \widehat D^\epsilon)$
    be the Chern form of $\mathcal F$ defined by \eqref{lagland}, and
    let $c_\ell(E, \widehat D^\epsilon)$ denote the component of
    degree $2\ell$.
    Let $\ell_1,\dots,\ell_m \in \N_{>0}$. Then
    \begin{equation} \label{eq:chernCurrent}
            c_{\ell_1}^{\Res}(E,D) \wedge \dots \wedge c_{\ell_m}^{\Res}(E,D) := 
        \lim_{\epsilon \to 0} c_{\ell_1}(E,\widehat{D}^\epsilon) \wedge
        \dots \wedge c_{\ell_m}(E,\widehat{D}^\epsilon)
    \end{equation}
      is a well-defined closed current, independent of $\chi_\epsilon$, that  represents
    $c_{\ell_1}(\mathcal F) \wedge \dots \wedge c_{\ell_m}(\mathcal{F})$
    and has support on $\supp \mathcal F$.
\end{theorem}

The Chern currents \eqref{eq:chernCurrent} are pseudomeromorphic in the sense of
\cite{AW2}, see Theorem \ref{thm:existence2}, which means that they have a
geometric nature similar to closed positive (or normal) currents, see
Section \ref{hemma}. We let
\begin{equation*}
    c^{\Res}(E,D) = 1 + c_1^{\Res}(E,D) + c_2^{\Res}(E,D) + \cdots.
\end{equation*}
The first
nontrivial component of $c^{\Res}(E,D)$ is (the current of
integration along) a cycle, see Theorem \ref{cykla} below.
We do not know whether Chern currents of higher degree are of order $0$ in general.

\begin{remark} \label{rem:bidegree}
If all the connections $D_k$ are $(1,0)$-connections, \textit{i.e.} the
$(0,1)$-part of each $D_k$ equals $\dbar$, then so are
the connections $\widehat{D}_k^\epsilon$. However, even if the $D_k$ are Chern connections, the $\widehat{D}_k^\epsilon$
are not Chern connections in general. Thus, it might be the case that the involved forms and currents in \eqref{eq:chernCurrent} contain terms of bidegree $(\ell + m,\ell-m)$ with $m > 0$
(but only when $\ell > \codim \mathcal{F}$ by Theorem~\ref{cykla} below).
\end{remark}

Our construction of Chern currents is inspired by the paper \cite{BB} by Baum and Bott,
where singular holomorphic foliations are studied by expressing characteristic classes
associated to a foliation as certain cohomological residues, more
precisely as push-forwards
of cohomology classes living in the singular set of the foliation.
A key point in the proofs in \cite{BB} are the concepts of connections compatible with and fitted to a complex of vector bundles.
One may check that their constructions of fitted connections
(with some minor adaptations) correspond to connections of the form \eqref{nordenstam}.
For the results in \cite{BB}, it was sufficient to consider Chern forms associated to connections \eqref{nordenstam} for $\epsilon$ small enough,
but fixed, while in the present paper, we study the limit of such forms when $\epsilon \to 0$.

\begin{example}\label{modren}
Let us compute $c^{\Res}(E,D)$ when $\mathcal
F$ is the structure sheaf $\Ok_Z$ of a divisor $Z\subset X$, defined by
a holomorphic section $s$ of a holomorphic line bundle $L$ over $X$, and
$(E,\varphi)$ is the locally free resolution
\begin{equation}\label{fadren}
        0 \to \Ok(L^*) \stackrel{s}{\to} \Ok\to \Ok_Z \to 0.
      \end{equation}
Assume that $L$ is equipped with a connection $D_L$; equip $E_1 = L^*$ with the
induced dual connection $D_{L^*}$, and $E_0$ with the trivial connection.
The minimal inverse of $s$ is $1/s$ and $D s = D_L s$,
so $\widehat D^\epsilon_1 = \chi_\epsilon (D_Ls/s) + D_{L^*}$, and  $\widehat
D^\epsilon_0$ is the trivial connection.
The curvature form of $\widehat D_1^\epsilon$ equals
$\widehat \Theta^\epsilon_1 = d\chi_\epsilon \wedge (D_Ls/s) -
(1-\chi_\epsilon)\Theta_L$,
where $\Theta_L$ is the curvature form of $D_L$ (which equals minus the curvature form of $D_{L^*}$).
By an appropriate formulation of the Poincar\'e-Lelong formula,
\begin{equation}\label{oklart}
  \lim_{\epsilon\to 0}d\chi_\epsilon \wedge \frac{D_Ls}{s} = 2\pi i [Z],
\end{equation}
where $[Z]$ is the current of integration along (the cycle of) $Z$.
Note that
\[
  c(E, \widehat D^\epsilon) = c(E_1, \widehat D_1^\epsilon)^{-1} = 1-
  c_1(E_1, \widehat D_1^\epsilon) + c_1(E_1,\widehat
  D_1^\epsilon)^2-\cdots.
\]
Thus, since $\Theta_L$ is smooth,
\begin{equation}\label{jogga}
  c^{\Res}_1(E, D) = \lim_{\epsilon\to 0} -c_1(E_1, \widehat
  D^\epsilon_1)
 = -\lim_{\epsilon\to 0}\frac{i}{2\pi}\widehat\Theta_1^\epsilon = 
  [Z].
\end{equation}
Thus, the first Chern current coincides
with $[Z]$, which can be seen as a canonical representative of $c_1(\Ok_Z)$ with
support on $\supp \Ok_Z$.

By further calculations of terms of all degrees, one can show that
\[
  c^{\Res}(E,D) = 1 + [Z](1 + c_1(L,D) + \dots + c_1^{n-1}(L,D)).
\]
\end{example}

Our next result is an explicit description of the first nontrivial
Chern current $c_p^{\Res}(E, D)$ in the case when $\mathcal F$ has pure
codimension $p$, \textit{i.e.} $\supp \mathcal F$ has pure dimension $\dim
X-p$, that generalizes \eqref{jogga}.
Recall that the \emph{(fundamental) cycle} of $\mathcal{F}$
is the cycle
\begin{equation}\label{eq:fundcycle}
    [\mathcal{F}] = \sum_i m_i [Z_i]
\end{equation}
(considered as a current of integration), where $Z_i$ are the
irreducible components of $\supp \mathcal{F}$, 
and $m_i$ is the \emph{geometric multiplicity} of $Z_i$ in
$\mathcal{F}$, see \emph{e.g.}\ \cite[Chapter~1.5]{Fulton}.

\begin{theorem} \label{cykla}
  Assume that $\mathcal{F}$ is a coherent analytic sheaf of pure
codimension $p>0$ that admits a locally free resolution $(E, \varphi)$
of the form \eqref{eq:Ecomplex}.  Moreover, assume that each $E_k$ is
equipped with a Hermitian metric and a $(1,0)$-connection $D_k$.  Then
\begin{equation*}
  c_p^{\Res}(E,D) = (-1)^{p-1}(p-1)![\mathcal F].
\end{equation*}
Moreover
\begin{equation}\label{hajka}
  c_\ell^{\Res}(E,D) = 0 \quad \text{ for } 0 < \ell < p
\end{equation}
and
\begin{equation}\label{bajka}
  c_{\ell_1}^{\Res}(E,D) \wedge \dots \wedge c_{\ell_m}^{\Res}(E,D) = 
  0  \quad \text{ for } m \geq 2 \text{ and } 0  < \ell_1 + \dots + \ell_m \leq p.
\end{equation}
Here the Chern currents on the left hand sides are defined by
\eqref{eq:chernCurrent}. 
\end{theorem}

In case $\mathcal{F}$ has codimension $p$, but not necessarily
\emph{pure} codimension $p$, then Theorem~\ref{cykla} still holds if
we  replace the first equation by
\begin{equation} \label{orenstrumpa}
        c_p^{\Res}(E,D) = (-1)^{p-1}(p-1)![\mathcal F]_p,
\end{equation}
where $[\mathcal F]_p$ denotes the part of $[\mathcal F]$ of codimension $p$,
\textit{i.e.} in \eqref{eq:fundcycle}, one only sums over the components $Z_i$ of codimension $p$.

In particular,  $c_p^{\Res}(E, D)$ is independent of the choice of
Hermitian metrics and $(1,0)$-connections on $(E,\varphi)$. Moreover, it follows that on cohomology level
\begin{equation}\label{janakippo}
c_p(\mathcal F) = (-1)^{p-1}(p-1)![\mathcal F],
  \end{equation}
where now the right hand side should be interpreted as a de Rham
class.
When $\mathcal{F}$ is the pushforward of a vector bundle from a
subvariety, that \eqref{janakippo} holds
is a well-known consequence of the Grothendieck-Riemann-Roch theorem, \emph{cf.}\
\cite[Examples~15.2.16 and~15.1.2]{Fulton}.

The proof of Theorem \ref{cykla} relies on a generalization of the
Poincar\'e-Lelong formula. Given a complex \eqref{eq:Ecomplex}
equipped with Hermitian metrics, Andersson and the second author
defined in \cite{AW1} an associated so-called residue current
$R^E = R = \sum R_k$ with support on $\supp \mathcal{F}$, where $R_k$ is
a $\Hom(E_0,E_k)$-valued $(0,k)$-current for $k = 0,\dots,N$, see
Section~\ref{ssect:AW1}.
The construction
involves the minimal
inverses $\sigma_k$ of $\varphi_k$. If $(E, \varphi)$ is the
complex \eqref{fadren}, then $R^E$ coincides with the residue current
$\dbar (1/s) = \lim_{\epsilon \to 0} \dbar \chi_\epsilon (1/s)$;  more
generally if
$(E,\varphi)$ is the Koszul complex of a complete intersection,
then $R^E$ coincides with the classical Coleff-Herrera residue
current, \cite{CH}.
Using residue currents, we can write the Poincar\'e-Lelong formula
\eqref{oklart} as
\begin{equation*}
  \dbar \frac{1}{s} \wedge D_Ls = 2\pi i [Z];
\end{equation*}
indeed, the left hand side in \eqref{oklart} equals $\lim
\dbar \chi_\epsilon \wedge (D_Ls)/s$.
Given a $\Hom(E_\ell,E_\ell)$-valued current $\alpha$, let $\tr \alpha$ denote the trace of $\alpha$.
In \cite{LW2, LW3} we proved the following generalization of the
Poincar\'e-Lelong formula:

\noindent
\emph{Assume that $R^E$ is the residue current associated with a
finite locally free resolution $(E,\varphi)$ of a coherent analytic
sheaf $\mathcal F$ of pure codimension $p$. Moreover assume that $D$
is a connection on $\End E$ induced by arbitrary $(1,0)$-connections
on $E_k$.  Then}
\begin{equation} \label{eq:residueCurrentFundCycle}
  \frac{1}{(2\pi i)^p p!} \tr (D\varphi_1 \cdots D\varphi_p R_p)
 = [\mathcal{F}].
\end{equation}
If $\mathcal F$ has codimension $p$, but not necessarily pure
codimenion, \eqref{eq:residueCurrentFundCycle} still holds if we
replace $[\mathcal F]$ by $[\mathcal F]_p$,
\emph{cf.} \cite[Theorem~1.5]{LW2}.  In view of this, Theorem \ref{cykla}, as
well as \eqref{orenstrumpa}, are direct consequences of the following
explicit description of (the components of low degree of) 
$c^{\Res}(E, D)$ in terms of $R^E$.

\begin{theorem} \label{thm:main}
 Assume that $\mathcal{F}$ is a coherent analytic sheaf of codimension
  $p>0$ that admits a locally free resolution
     $(E,\varphi)$ of
     the form \eqref{eq:Ecomplex}.
    Moreover, assume that each $E_k$ is equipped with a Hermitian
    metric and a $(1,0)$-connection $D_k$.
Let $R$ be the associated residue current and $D$ the connection on $\End
E$ induced by the $D_k$.
Then
    \begin{equation*}
        c_p^{\Res}(E,D) = \frac{(-1)^{p-1}}{(2\pi i)^p p} \tr
        (D\varphi_1 \cdots D\varphi_p R_p).
      \end{equation*}
      Moreover \eqref{hajka} and \eqref{bajka} hold.
\end{theorem}

\smallskip

In fact, we formulate and prove our results in a slightly more general
setting. We consider the Chern class $c(E)$ of a generically exact complex of vector
bundles $(E, \varphi)$
that is not necessarily a locally free resolution of a coherent
sheaf. Theorem~\ref{thm:existence2} below asserts that $c(E, \widehat D^\epsilon)$ as well
as products of such currents have well-defined limits when $\epsilon \to 0$ and represent the corresponding (products of) Chern
classes. In particular, Theorem~\ref{thm:existence} follows.
In Theorem \ref{thm:mainGeneral}, if $(E, \varphi)$ is exact outside a variety of codimension $p$, we give an explicit description of
$c_p^{\Res}(E, D):= \lim_{\epsilon\to 0} c_p(E, \widehat D^\epsilon)$
terms of residue currents that
generalizes Theorem~\ref{thm:main}.
From this and a more general version of the Poincar\'e-Lelong formula
\eqref{eq:residueCurrentFundCycle} it follows that
if the cohomology groups are of pure codimension $p$, then
$c_p^{\Res}(E, D) = (-1)^{p-1}(p-1)![E]$, where $[E]$ is the cycle of
$(E, \varphi)$, see Corollary~\ref{lydia} and \eqref{cycleofcomplex}.

\smallskip
Our results could alternatively be formulated in term of the Chern
character $\ch(E)$ of $E$. From Theorem~\ref{thm:existence}, for
$\ell>0$, we obtain
a current $\ch^{\Res}_\ell(E, D)$ that represents the $\ell^\mathrm{th}$ graded
piece $\ch_\ell(E)$ of the Chern character, see Section
\ref{chernsec}.
Theorems \ref{cykla} and \ref{thm:main} are then equivalent to
    \begin{align*}
        \ch_p^{\Res}(E,D) &= \frac{1}{(2\pi i)^p p!} \tr(D\varphi_1 \cdots D\varphi_p R_p) = [\mathcal F],\\
   \ch_\ell^{\Res}(E,D) &= 0 \quad \text{ for } \ell < p, \ \text{and}\\
   \ch_{\ell_1}^{\Res}(E,D) \wedge \dots \wedge \ch_{\ell_m}^{\Res}(E,D) &= 0
   \quad \text{ for } m \geq 2 \ \text{and} \ \ell_1 + \dots + \ell_m
   \leq p,
 \end{align*}
 see Theorem \ref{saltkar} and Remark \ref{badkar}.

\smallskip

We refer to the currents in Theorem~\ref{thm:existence} as Chern
currents, in analogy with the usual Chern forms representing Chern
classes.  In works of Bismut, Gillet, and Soulé \cite{BGS4,BGS5}
appears the similarly named concept of Bott-Chern currents, that are
certain explicit $dd^c$-potentials in a transgresssion formula in a
Grothendieck-Riemann-Roch theorem, and not directly related to our
currents. 

There are some similarities between our results and results by Harvey and
Lawson. In \cite{HL} they study characteristic classes of
morphisms $\varphi:E_0\to E_1$ of vector bundles, and only in
very special situations there is overlap between their results and
ours. We remark that the connection \eqref{nordenstam} that plays
a crucial role in our work essentially appears and is important
in \cite{HL}, see, in particular, \cite[Section I.4]{HL}.

Chern classes of coherent sheaves, without the assumption of the
existence of a global locally free resolution, were studied in the
thesis of Green, \cite{Green}, as well as in various recent papers,
including \cite{Gri,Hos,Hos2,Qiang,BSW,Wu}. Several of these papers
also concern classes in finer cohomology theories than de Rham
cohomology, as for example (rational or complex) Bott-Chern or Deligne
cohomology. 

In the present paper, our focus has been to find explicit
representatives of Chern classes of a coherent sheaf with support on
the support of the sheaf, a type of result which as far as we can
tell, none of the above mentioned works seems to consider. By
incorporating the construction of residue currents associated with a
twisted resolution from \cite{JL}, it might be possible to extend our
results to arbitrary coherent sheaves, without any assumptions about
the existence of a global locally free resolution. We plan to explore
this in future work. The currents we study provide representatives of
the Chern classes in de Rham cohomology. Our methods unfortunately do
not seem to yield representatives in the finer cohomology theories
mentioned above, since for example Chern classes in complex Bott-Chern
cohomology as in \cite{Qiang,BSW}, are naturally obtained from Chern
forms of the Chern connection of a hermitian metric, while our
construction, building on the techniques in \cite{BB}, involve Chern
forms of connections that are not Chern connections of a hermitian
metric. 

\smallskip

The paper is organized as follows. In Section \ref{stromprelim}
we give some necessary background on (residue) currents.
In Section \ref{trean} we describe
Chern forms and Chern characters,
and in Section \ref{taget}
we discuss compatible connections. The proofs of (the generalized
versions of) Theorems
\ref{thm:existence} and \ref{thm:main} occupy Sections \ref{gula} and
\ref{chernsec}, respectively. Finally in Section \ref{triumf} we
compute $c^{\Res}(E, D)$ for an explicit choice of a locally free
resolution $(E,\varphi)$ of a coherent sheaf $\mathcal F$. In particular, we
compute $c^{\Res}_\ell(E, D)$ for $\ell>\codim \mathcal F$ in this case.

\subsection*{Acknowledgements}

This paper is very much inspired by an ongoing joint project with
Lucas Kaufmann, which aims to understand Baum-Bott residues
in terms of (residue) currents. We are greatly
indebted to him for many valuable discussions on this topic. We would
also like to thank Dennis Eriksson for many important discussions and
helpful comments on a previous version of this paper.

\section{Currents associated with complexes of vector bundles}\label{stromprelim}

We say that a function $\chi:\R_{\geq 0}\to \R_{\geq 0}$ is a \emph{smooth approximand
  of the characteristic function} $\chi_{[1,\infty)}$ of the interval
  $[1,\infty)$ and write
  \[
    \chi \sim \chi_{[1,\infty)}
  \]
 if $\chi$ is smooth and $\chi(t) \equiv 0$ for $t \ll 1$ and
$\chi(t) \equiv 1$ for $t \gg 1$.
Note that if $\chi \sim \chi_{[1,\infty)}$ and $\hat \chi = \chi^\ell$,
then $\hat \chi \sim \chi_{[1,\infty)}$ and
\begin{equation}\label{information}
d\hat\chi = \ell\chi^{\ell-1} d\chi.
\end{equation}

\subsection{Pseudomeromorphic currents}\label{hemma}

Let $f$ be a (generically nonvanishing) holomorphic function on a
(connected)
complex manifold $X$. 
Herrera and Lieberman \cite{HeLi}, proved that the \emph{principal value}
\begin{equation*}
\lim_{\epsilon\to 0}\int_{|f|^2>\epsilon}\frac{\xi}{f}
\end{equation*}
exists for test forms $\xi$ and defines a current $[1/f]$.
It follows that $\dbar[1/f]$ is a current with support on the zero set
$Z(f)$ of $f$; such a current is called a \emph{residue current}.
Assume that $\chi\sim\chi_{[1,\infty)}$ and that
 $F$ is a (generically
nonvanishing) section of a Hermitian vector bundle such that
$Z(f)\subseteq \{ F = 0 \}$.
Then
\begin{equation}\label{strunta}
  [1/f] = \lim_{\epsilon\to 0} \frac{\chi(|F|^2/\epsilon)}f ~~~~~ \text{
    and } ~~~~~
  \dbar[1/f] = \lim_{\epsilon\to 0} \frac{\dbar\chi(|F|^2/\epsilon)}f,
  \end{equation}
  see \emph{e.g.} \cite{AW3}.
  In particular, the limits are independent of $\chi$ and $F$.

In the literature there are various generalizations of residue
currents and principal value currents.
In particular, Coleff and Herrera \cite{CH} introduced
products like
\begin{equation}\label{apan2}
[1/f_1]\cdots [1/f_r] \dbar[1/f_{r + 1}]\wedge \cdots
\wedge \dbar[1/f_m].
\end{equation}
In order to obtain a coherent approach to questions about residue and
principal value currents was introduced in \cite{AW2} the sheaf
$\PM_X$ of {\it pseudomeromorphic currents} on $X$, consisting of
direct images under holomorphic mappings of products of test forms and
currents like \eqref{apan2}. See \emph{e.g.} \cite[Section~2.1]{AW3} for a
precise definition; in particular it follows from the definition that
$\PM$ is closed under push-forwards of modifications. Also, we refer
to \cite{AW3} for the results mentioned in this subsection.  The sheaf
$\PM_X$ is closed under $\dbar$ and under multiplication by smooth
forms.  Pseudomeromorphic currents have a geometric nature, similar to
closed positive (or normal) currents.  For example, the {\it dimension
principle} states that if the pseudomeromorphic current $\mu$ has
bidegree $(*,p)$ and support on a variety of codimension strictly
larger than $p$, then $\mu$ vanishes.

The sheaf $\PM_X$ admits natural restrictions to constructible subsets.
In particular, if $W$ is a subvariety of the open subset $\mathcal{U}\subseteq X$,
and $F$ is a section of a vector bundle such that $\{F = 0\} = W$, then the restriction
to $\mathcal{U}\setminus W$ of a pseudomeromorphic current $\mu$ on $\mathcal{U}$ is
the pseudomeromorphic current 
\[
    \1_{\mathcal{U}\setminus W} \mu := \lim_{\epsilon \to 0} \chi(|F|^2/\epsilon) \mu|_{\mathcal{U}},
\]
where $\chi \sim \chi_{[1,\infty)}$ as above.
This definition is independent of the choice of $F$ and $\chi$.

A pseudomeromorphic current $\mu$ on $X$ is said to have the \emph{standard extension property} (SEP)
if $\1_{\mathcal{U}\setminus W} \mu = \mu|_{\mathcal{U}}$
for any subvariety $W \subseteq \mathcal{U}$ of positive codimension, where
$\mathcal{U} \subseteq X$ is any open subset.
By definition, it follows that if $\mu$ has the SEP and $F\not\equiv 0$ is any holomorphic section of a vector bundle,
then
\begin{equation} \label{eq:SEP}
    \lim_{\epsilon\to 0} \chi(|F|^2/\epsilon) \mu = \mu.
\end{equation}

\subsection{Superstructure and connections on a complex of vector bundles}\label{smaragd}

Let $(E,\varphi)$ be a complex
\begin{equation}\label{turkos}
  0 \xrightarrow[]{} E_N
  \xrightarrow[]{\varphi_N} E_{N-1}
  \xrightarrow[]{\varphi_{N-1}}
  \cdots \xrightarrow[]{\varphi_2}
  E_1 \xrightarrow[]{\varphi_1} E_0 \xrightarrow[]{} 0,
\end{equation}
of vector bundles over $X$.  As in \cite{AW1}, see also \cite{LW2}, we
will consider the complex $(E,\varphi)$ to be equipped with a
so-called superstructure, \textit{i.e.} a $\mathbb{Z}_2$-grading, which splits
$E:= \oplus E_k$ into odd and even parts $E^ + $ and $E^-$, where 
$E^ + = \oplus E_{2k}$ and $E^- = \oplus E_{2k + 1}$.
Also $\End E$ gets a superstructure by letting the even part be
the endomorphisms preserving the degree, and the odd part the
endomorphisms switching degrees.

This superstructure affects how form- and current-valued endomorphisms
act.  Assume that $\alpha = \omega\otimes \gamma$ is a section of
$\E^\bullet (\End E)$, where $\gamma$ is a holomorphic section of
$\Hom(E_\ell,E_k)$, and $\omega$ is a smooth form of degree $m$.  Then
we let $\deg_f \alpha = m$ and $\deg_e \alpha = k-\ell$ denote the
\emph{form} and \emph{endomorphism} degrees, respectively, of
$\alpha$.  The \emph{total} degree is 
$\deg \alpha = \deg_f \alpha + \deg_e \alpha$.  If $\beta$ is a
form-valued section of $E$, \textit{i.e.} $\beta = \eta\otimes\xi$, where $\eta$
is a scalar form, and $\xi$ is a section of $E$, both homogeneous in
degree, then the action of $\alpha$ on $\beta$ is defined by
\begin{equation}
  \label{eq:superaction}
  \alpha(\beta)  := (-1)^{(\deg_e \alpha)(\deg_f \beta)}
  \omega \wedge \eta \otimes \gamma(\xi).
\end{equation}
If furthermore, $\alpha' = \omega'\otimes \gamma'$, where $\gamma'$ is a holomorphic
section of $\End E$, and $\omega'$ is a smooth form, both homogeneous in degree, then
we define
\begin{equation*}
    \alpha \alpha' := (-1)^{(\deg_e \alpha)(\deg_f \alpha')} \omega\wedge \omega' \otimes \gamma \circ \gamma'.
\end{equation*}
For an $(m \times n)$-matrix $A$ and an $(n \times m)$-matrix $B$, we have that $\tr (AB) = \tr (BA)$, while for the morphisms $\alpha$ and $\alpha'$ above,
we get such an equality with a sign due to the superstructure,
\begin{equation}\label{spar}
    \tr (\alpha\alpha') = (-1)^{(\deg \alpha)(\deg \alpha')-(\deg_e \alpha)(\deg_e \alpha')}\tr (\alpha'\alpha),
\end{equation}
see \cite[Equation (2.14)]{LW2}.

Note that $\dbar$ extends in a way that respects the superstructure to act on $\End E$-valued morphisms.
In particular,
\begin{equation}\label{eq:dbarLeibniz}
\dbar(\alpha \alpha') = \dbar \alpha \alpha' + (-1)^{\deg \alpha} \alpha \dbar \alpha'.
\end{equation}

We will consider the situation when $(E, \varphi)$ is equipped with a
 connection $D = D_E = (D_0, \ldots, D_N)$, where $D_k$ is a connection on
$E_k$.
Then there is an induced connection $\oplus D_k$ on $E$, that we also denote
by $D_E$. This in turn induces a connection $D_{\End}$
on $\End E$ that takes the superstructure into account,  defined by 
\begin{equation}\label{DEnd}
D_{\End} \alpha := D_E \circ \alpha - (-1)^{\deg \alpha} \alpha \circ
D_E,
\end{equation}
if $\alpha$ is a $\End E$-valued form.
It satisfies the following Leibniz rule, \cite[Equation (2.4)]{LW2},
\emph{cf.}\ \eqref{eq:dbarLeibniz}
\begin{equation}\label{bord}
D_{\End} (\alpha \alpha') = D_{\End} \alpha \alpha' + (-1)^{\deg \alpha}
\alpha D_{\End} \alpha'.
\end{equation}
To simplify notation, we will sometimes drop the subscript $\End$ and simply
denote this connection by $D$. If $\Theta_k$ denotes the curvature
form of $D_k$, and
$\alpha:E_k\to E_\ell$, then, by \eqref{DEnd},
\begin{equation}\label{karlsborg}
  DD\alpha = 
  \Theta_\ell
  \alpha + (-1)^{\deg \alpha + \deg \alpha
    + 1}  \alpha
  \Theta_k
  = 
  \Theta_\ell\alpha - \alpha\Theta_k.
\end{equation}
The above formulas hold also when $\alpha$ and $\alpha'$ are current-valued instead of form-valued,
as long as the involved products of currents are well-defined.

We let $D_k'$ and $D_k''$ denote the $(1,0)$- and $(0,1)$-parts of
$D_k$, respectively, and we let $D' = (D_k')$ and $D'' = (D''_k)$ denote the
corresponding $(1,0)$- and $(0,1)$-parts of $D_E = (D_k)$.
We say that $D_E$ is a \emph{$(1,0)$-connection} if each $D_k$ is a
$(1,0)$-connection, \textit{i.e.} $D''_k = \dbar$.
We will use the following consequence of \eqref{karlsborg}:
assume that $D_E$ is a $(1,0)$-connection, and $\alpha
: E_k \to E_\ell$ is a holomorphic (or more generally a $\dbar$-closed
form-valued) morphism.
Then
\begin{equation} \label{eq:dbarD}
      \dbar D\alpha = (\Theta_\ell)_{(1,1)} \alpha - \alpha
      (\Theta_k)_{(1,1)},
\end{equation}
where $(\cdot)_{(1,1)}$ denotes the component of bidegree
$(1,1)$.

Since $(E,\varphi)$ is a complex and $\varphi_k$ has odd degree, it follows
from \eqref{bord} that
\begin{equation}\label{exakt}
\varphi_{k-1}D\varphi_k = D\varphi_{k-1} \varphi_k.
\end{equation}

\subsection{Residue currents associated to a
  complex} \label{ssect:AW1}

Let us briefly recall the construction in \cite{AW1}.
Assume that we have a generically exact complex $(E, \varphi)$ of
vector bundles over a complex manifold $X$ of
the form \eqref{turkos},
and assume that each $E_k$ is equipped with some Hermitian metric. If
$Z_k$ is the analytic set where $\varphi_k$
has lower rank than its generic rank, then outside of $Z_k$
the minimal (or Moore-Penrose) inverse
$\sigma_k:E_{k-1}\to E_k$ of $\varphi_k$ is determined by the following properties:
$\varphi_k \sigma_k \varphi_k = \varphi_k$, $\im \sigma_k \perp \im \varphi_{k + 1}$,
and $\sigma_{k + 1} \sigma_k = 0$.
One can verify that $\sigma_k$ is smooth outside of $Z_k$.
Since $\sigma_{k}\sigma_{k-1} = 0$ and
$\sigma_k$ has odd degree, by \eqref{eq:dbarLeibniz},
\begin{equation}\label{dexakt}
\sigma_{k}\dbar \sigma_{k-1} = \dbar \sigma_{k}\sigma_{k-1}.
\end{equation}
Let $Z$ be the set where $(E,\varphi)$ is not pointwise exact. It follows from the definition of $\sigma_k$ that
\begin{equation}\label{identitet}
   \varphi_k\sigma_k + \sigma_{k-1}\varphi_{k-1} = \Id_{E_{k-1}}
\end{equation}
outside $Z$, or more generally outside $Z_k \cup Z_{k-1}$.
Applying \eqref{eq:dbarLeibniz} to \eqref{identitet}, we obtain that
outside $Z$
\begin{equation}\label{dstreck}
\varphi_k\dbar\sigma_k = \dbar \sigma_{k-1}\varphi_{k-1}
\end{equation}
and furthermore applying \eqref{bord} to this equality, we get that
\begin{equation}\label{kalkon}
D\varphi_k\dbar\sigma_k = D\dbar \sigma_{k-1}\varphi_{k-1} + \dbar \sigma_{k-1} D\varphi_{k-1}
 + 
\varphi_k D\dbar \sigma_k.
\end{equation}

\begin{lemma} \label{lma:PMexistence}
    Let $X$, $(E,\varphi)$, $Z$, and $\sigma_k$ be as above. Assume that
    for each $j = 1, \ldots, m$, $s_j$ is an entry of $\sigma_k$, $\partial
    \sigma_k$, or $\dbar\sigma_k$ for some $k$ in
    some local trivialization, and let $s = s_1\cdots s_m$.
    Assume that $\chi\sim\chi_{[1,\infty)}$ and that $F$ is a (generically nonvanishing) holomorphic section of a
    vector bundle over $X$ such that $Z\subset\{ F = 0 \}$.
    Then the limits
    \begin{equation*}
      \lim_{\epsilon\to 0} \chi(|F|^2/\epsilon) s
      \qquad\text{and}\qquad
      \lim_{\epsilon\to 0} \dbar\chi(|F|^2/\epsilon)\wedge s
    \end{equation*}
    exist and define pseudomeromorphic currents on $X$ that are independent
    of the choices of $\chi$ and $F$; the support of the second
    current is contained in $Z$.
    Furthermore,
    \begin{equation*}
        \lim_{\epsilon\to 0} \partial\chi(|F|^2/\epsilon)\wedge s = 0.
    \end{equation*}
  \end{lemma}

\begin{proof}
  By Hironaka's theorem there is a holomorphic modification 
$\pi : \tilde{X} \to X$, such that for each $k$, $\pi^*\sigma_k$ is
locally of the form $(1/\gamma_k) \tilde \sigma_k$, where $\gamma_k$
is holomorphic with $Z(\gamma_k)\subset \tilde Z:= \pi^{-1}Z$, and
$\tilde\sigma_k$ is smooth, see \cite[Section~2]{AW1}.  Now, where
$\chi(|\pi^* F|^2/\epsilon)\not\equiv 0$, 
$\dbar \pi^* \sigma_k = (1/\gamma_k) \dbar \tilde \sigma_k$ and
$\partial \pi^* \sigma_k = \partial (1/\gamma_k) \tilde \sigma_k + 
(1/\gamma_k) \partial \tilde \sigma_k$.  
Since each holomorphic derivative 
${\partial/\partial z_i} (1/\gamma_k)$ is a meromorphic function with
poles contained in $\tilde Z$ it follows that $\pi^*s_j$ equals (a sum
of terms of the form) $(1/g_j)\tilde s_j$, where $g_j$ is holomorphic
with $Z(g_j)\subset \tilde Z$, and $\tilde s_j$ is smooth. Thus
$\pi^*s$ equals (a sum of terms of the form) $(1/g)\tilde s$, where
$g$ is holomorphic with $Z(g)\subset \tilde Z$, and $\tilde s$ is
smooth.  In view of \eqref{strunta},
\begin{equation*}
  \lim_{\epsilon\to 0} \chi(|\pi^*F|^2/\epsilon) \pi^*s
  \quad\text{and}\quad
  \lim_{\epsilon\to 0}
  \dbar\chi(|\pi^* F|^2/\epsilon)\wedge \pi^*s
\end{equation*}
are well-defined pseudomeromorphic currents on $\tilde X$ independent
of $\chi$ and $F$; the second current has support on $\tilde Z$. Since
$\PM$ is closed under push-forwards of modifications, \emph{cf.}\ Section
\ref{hemma}, this proves the first part of the lemma.

As proved above, the limit
 \begin{equation*}
      \mu := \lim_{\epsilon\to 0} \chi(|F|^2/\epsilon) s
 \end{equation*}
exists. This current is in fact a so-called almost semi-meromorphic current,
\emph{cf.} \cite[Section~4]{AW3}, and in particular, it has the SEP.
By \cite[Theorem~3.7]{AW3},  $\partial\mu$ also has the SEP.
Thus,
\begin{equation*}
  \lim_{\epsilon\to 0 } \partial\chi(|F|^2/\epsilon) \wedge \mu = 
  \lim_{\epsilon\to 0 }
  \partial(\chi(|F|^2/\epsilon) \wedge \mu)
  -\lim_{\epsilon\to 0 } \chi(|F|^2/\epsilon)\partial\mu
 = \partial\mu-\partial \mu = 0,
\end{equation*}
which proves the last part of the lemma. Here, in the second equality, we
have used that the two limits exist and are both equal to $\partial\mu$
by \eqref{eq:SEP}.
\end{proof}

In particular
\begin{equation} \label{eq:Rdef}
    R^\ell_k := \lim_{\epsilon \to 0} \dbar \chi(|F|^2/\epsilon)
    \wedge \sigma_{k} \dbar\sigma_{k-1} \cdots \dbar\sigma_{\ell + 1}
\end{equation}
is a $\Hom (E_\ell, E_k)$-valued pseudomeromorphic current of bidegree
$(0,k-\ell)$ with support on $Z$; in fact, it follows from the proof
that the support is contained in $Z _{\ell + 1}\cup \cdots\cup Z_k$.  If
$\ell = k-1$, then the right hand side of \eqref{eq:Rdef} should be
interpreted as 
$\lim_{\epsilon \to 0} \dbar \chi(|F|^2/\epsilon) \wedge \sigma_{k}$.
The residue current $R^E = R:= \sum R_k^\ell$ associated with 
$(E, \varphi)$ was introduced in \cite{AW1}, \emph{cf.}\ the introduction.
Assume that $(E, \varphi)$ is a locally free resolution of a coherent
analytic sheaf $\mathcal F$.  Then $R_k^\ell$ vanishes for $\ell>0$ by
\cite[Theorem~3.1]{AW1}. In this case $R = \sum R_k$, where
$R_k = R_k^0$.

\smallskip
Given a complex $(E, \varphi)$ of vector bundles of the form
\eqref{turkos}, following \cite{LW3}, we define the \emph{cycle}
\begin{equation} \label{cycleofcomplex}
    [E] = \sum_{k = 0}^N (-1)^k [\mathcal{H}_k(E)],
  \end{equation}
  where $\mathcal H_k$ is the homology sheaf of $(E,\varphi)$ at level
  $k$.
Note that if $(E,\varphi)$ is a locally free resolution of a coherent
analytic sheaf $\mathcal{F}$, then
$[E] = [\mathcal{F}]$.
In \cite{LW3} we prove the following generalization of \eqref{eq:residueCurrentFundCycle}.

\begin{theorem} \label{thm:LW3}
   Let $(E,\varphi)$ be a complex of Hermitian vector bundles of the
   form
   \eqref{turkos} such that $\mathcal H_k(E)$ has pure codimension
   $p>0$ or vanishes, for $k = 0,\dots,N$, and let $D$ be an arbitrary
   $(1,0)$-connection on
   $(E, \varphi)$.
Then,
    \begin{equation*}
        \frac{1}{(2\pi i)^p p!} \sum_{k = 0}^{N-p}
 (-1)^k \tr (D\varphi_{k + 1} \cdots D\varphi_{k + p} R^{k}_{k + p}) = [E].
    \end{equation*}
\end{theorem}

\section{Chern forms and Chern characters}\label{trean}

\subsection{Chern classes and forms}\label{kalla}

Assume that $E$ is a holomorphic vector bundle of rank $r$ equipped with a connection
$D$.  Then recall that the (total) Chern form
$c(E, D) = 1 + c_1(E, D) + \cdots + c_r(E, D)$ is defined by
\[
  \sum_{\ell = 0}^r c_\ell(E, D) t^\ell = \det \left (I + \frac{i}{2\pi}
  \Theta t \right),
  \]
where $\Theta$ is the curvature matrix of $D$ in a local
trivialization; in particular, $c_\ell(E, D)$ is a form of degree
$2\ell$. The de~Rham cohomology class of $c(E, D)$ is the
\emph{(total) Chern class} $c(E) = \sum c_\ell(E)$ of the vector bundle $E$.

\smallskip

If $(E, \varphi)$ is a complex of vector bundles of the form
\eqref{turkos} that is not necessarily a locally free resolution of a
coherent analytic sheaf, in line with the Chern theory of virtual
bundles as in \emph{e.g.} \cite[Section 4]{BB} or \cite[Section II.8.C]{Su}, we let
\begin{equation*}
    c(E) = \prod_{k = 0}^N c(E_k)^{(-1)^k}.
\end{equation*}
Moreover, if $(E, \varphi)$ is equipped with a connection $D = (D_k)$, \emph{cf.}\ Section \ref{smaragd}, we
let
\begin{equation} \label{syster}
    c(E,D) = \prod_{k = 0}^N c(E_k,D_k)^{(-1)^k}
  \end{equation}
  and we let 
$c_\ell(E, D) = c(E,D)_\ell$ be the component of degree
$2\ell$.

Consider now a coherent analytic sheaf $\mathcal{F}$ with a
locally free resolution \eqref{eq:Ecomplex}. We define the Chern class
of $\mathcal F$ by \eqref{eq:cClassSheafDef}, \textit{i.e.} $c(\mathcal{F}) = c(E)$,
and if $(E,\varphi)$ is equipped with a connection $D$, then this class may be
represented by \eqref{syster}.
This definition of Chern classes of coherent sheaves may be motivated in terms of K-theory.
However, it 
is typically considered only on manifolds
with the so-called resolution property.
Recall that a complex manifold $X$ is said to have the \emph{resolution property} if
any coherent analytic sheaf $\mathcal{F}$ on $X$ has a finite locally free resolution
\eqref{eq:Ecomplex}.
For such manifolds, the definition \eqref{eq:cClassSheafDef} is the unique extension
of the definition of Chern classes from locally free sheaves to
coherent analytic sheaves that satisfies the following \emph{Whitney formula}: if $0\to \mathcal{F}' \to \mathcal{F} \to \mathcal{F}'' \to 0$ is a short exact sequence of sheaves,
then $c(\mathcal{F}) = c(\mathcal{F}'')c(\mathcal{F}'')$,
\emph{cf.} \cite[Th\'eor\`eme 2]{BS} or \cite[Chapter~14.2]{EH}.

In this paper, we define Chern classes of coherent sheaves by \eqref{eq:cClassSheafDef} also on manifolds which do not have the resolution property, but then necessarily only for coherent sheaves
with a locally free resolution \eqref{eq:Ecomplex}.
Note that if we are on a manifold for which the resolution property does not hold, it is not
immediate that the de Rham cohomology class of \eqref{eq:cClassSheafDef} is well-defined,
\textit{i.e.} independent of the resolution. However, that it is well-defined follows from a construction of Chern classes
of arbitrary coherent analytic sheaves on arbitrary complex manifolds by Green, \cite{Green}, see also \cite{TTGreen},
since in case one has a global locally free resolution of finite length, the definition in \cite{Green} coincides with
the one in \eqref{eq:cClassSheafDef}.

\subsection{The Chern character (form) of a vector bundle}\label{hemmavid}

Assume that $E$ is a holomorphic vector
bundle of rank $r$.  Then formally we can write
\begin{equation*}
  1 + c_1(E)t + \cdots + c_r(E) t^r = \prod_{i = 1}^r (1 + \alpha_i t),
  \end{equation*}
where $\alpha_i$ are the so-called \emph{Chern roots} of $E$, see \emph{e.g.}
\cite[Remark~3.2.3]{Fulton}.
In particular, this means that $c_\ell(E) = e_\ell(\alpha_1,\ldots,
\alpha_r)$, where $e_\ell$ is the $\ell^\mathrm{th}$ \emph{elementary symmetric
  polynomial}
\begin{equation*}
e_\ell(x) = e_\ell(x_1,\ldots, x_r) = \sum_{1 \leq i_1 < \dots < i_\ell \leq r} x_{i_1} \cdots x_{i_\ell}.
\end{equation*}
The \emph{Chern character} of $E$ may formally be defined as the symmetric polynomial $\ch(E) = 
\sum_{i = 1}^r e^{\alpha_i}$ in the
Chern roots, see \emph{e.g.} \cite[Example ~3.2.3]{Fulton}. In particular, the $\ell^\mathrm{th}$
graded piece is
\begin{equation}\label{tetrapak}
  \ch_\ell(E) = \frac{1}{\ell!} p_\ell (\alpha_1,\ldots, \alpha_r),
\end{equation}
where $p_\ell$ is the $\ell^\mathrm{th}$ \emph{power sum polynomial}
\begin{equation*}
p_\ell(x) = p_\ell(x_1,\ldots, x_r) = \sum_{i = 1}^r x_i^\ell.
\end{equation*}
Since any symmetric polynomial in $x_i$ may be expressed as a unique
polynomial in $e_j(x)$, there are polynomials 
$Q_\ell(t_1, \ldots, t_\ell)$, $\ell\geq 1$, such that
$p_\ell(x) = Q_\ell(e_1(x), \ldots, e_\ell(x))$; these are sometimes
called \emph{Hirzebruch--Newton polynomials}.  If $t_j$ is given weight
$j$, then $Q_\ell(t_1, \ldots, t_\ell)$ is a weighted homogenous
polynomial of degree $\ell$.  Written out explicitly, Definition
\eqref{tetrapak} should be read as
\begin{equation*}
  \ch_\ell(E) = \frac{1}{\ell!} Q_\ell\big(c_1(E), \ldots, c_\ell(E)\big).
\end{equation*}

If $E$ is equipped with a connection $D$, one can analogously
define \emph{Chern character forms}
\begin{equation}\label{havre}
  \ch_\ell(E, D)
 = \frac{1}{\ell!} Q_\ell\big(c_1(E, D), \ldots, c_\ell(E, D)\big)
\end{equation}
and $\ch(E, D) = \sum \ch_\ell(E, D)$
representing the
Chern character.
If $\Theta$ is the curvature corresponding to $D$ (in a local
trivialization), then
\begin{equation} \label{lovande}
  \ch_\ell(E,D)
 = \frac{1}{\ell!}\tr \left(\frac{i}{2\pi}\Theta\right)^\ell,
\end{equation}
\emph{cf.} \emph{e.g.} \cite[\S B.4-6]{Tu}.

The polynomials $Q_\ell$ may be computed recursively through Newton's
identities, 
\begin{equation} \label{newtonid}
  p_\ell(x) = (-1)^{\ell-1} \ell e_\ell(x) + \sum_{i = 1}^{\ell-1}
  (-1)^{\ell-i-1}e_{\ell-i}(x) p_i(x), \quad \ell\geq 1.
\end{equation}
In particular, it follows that the $Q_\ell$ are independent of $r$.
Moreover, $\ch_\ell(E, D)$ is of the form
\begin{equation} \label{sommar}
  \ch_\ell(E, D) = \frac{(-1)^{\ell-1}}{(\ell-1)!} c_\ell(E, D) + 
  \widetilde{Q}_\ell\big(c_1(E, D),\ldots,c_{\ell-1}(E, D)\big),
\end{equation}
where $\widetilde{Q}_\ell$ is a weighted homogeneous polynomial of degree
$\ell$, and conversely,
\begin{equation} \label{sommarbad}
  c_\ell(E, D) = (-1)^{\ell-1} (\ell-1)! \ch_\ell(E, D) + \widehat{Q}_\ell\big
  (\ch_1(E, D),\ldots,\ch_{\ell-1}(E, D)\big),
\end{equation}
where $\widehat{Q}_\ell$ is a weighted homogeneous polynomial of degree
$\ell$.

\begin{example}
  We obtain from \eqref{newtonid} that 
$p_1 = e_1 \text{ and } p_2 = e_1^2-2e_2$.  Thus, $Q_1(t_1) = t_1$ and
$Q_2(t_1,t_2) = t_1^2-2t_2$, so
\begin{equation} \label{eq:ch12}
  \ch_1(E,D) = c_1(E,D)
  \quad \text{ and } \quad
  \ch_2(E,D) = \frac{1}{2}(c_1(E,D)^2-2c_2(E,D)).
\end{equation}
\end{example}

We have the following (formal) relationship between $e_\ell(x)$ and
$p_\ell(x)$, and thus $Q_\ell(e_1,\ldots, e_\ell)$:
\begin{equation}\label{kaffekopp}
        \ln \left (\sum_{\ell\geq 0} e_\ell(x) t^\ell \right) = 
        \sum_{\ell\geq 1}  \frac{(-1)^{\ell-1}}{\ell} p_\ell(x)t^\ell
        = 
        \sum_{\ell\geq 1}  \frac{(-1)^{\ell-1}}{\ell} Q_\ell(e_1,
        \ldots, e_\ell)t^\ell.
      \end{equation}
This follows \emph{e.g.} by integrating
\cite[Chapter~I, Equation~(2.10')]{MacDonald}
with respect to $t$.
Since $e_1, \ldots, e_r$ are algebraically independent, \eqref{kaffekopp} holds if we replace the $e_\ell$ by $a_\ell$ in any
commutative ring. In particular, if we apply \eqref{kaffekopp} to $e_\ell = c_\ell(E,
D)$ and take the components of degree $2\ell$ (the coefficents of $t^\ell$) we
get
\begin{equation}\label{veckoslut}
        \ln \big(c(E, D) \big)_\ell = 
     (-1)^{\ell-1}(\ell-1)! \ch_\ell(E, D);
\end{equation}
here $()_\ell$ denotes the part of form degree $2\ell$.

\subsection{The Chern character of a complex of vector
  bundles}\label{glassbil}

Let $(E, \varphi)$ be a complex of vector bundles of the form
\eqref{turkos}.
Then the Chern character can be defined as
\begin{equation*}
  \ch(E) = \sum_{k = 0}^N (-1)^k \ch(E_k),
\end{equation*}
\emph{cf.}, \emph{e.g.}, \cite[Chapter 14.2.1]{EH} and
\cite[Chapter~V.3]{Karoubi}.

If $(E, \varphi)$ is equipped with a
connection $D = (D_k)$, for $\ell\geq 1$ we define a \emph{Chern
  character form}
$\ch_\ell(E, D)$ through
\eqref{havre}.
Then $\ch_\ell(E, D)$ inherits properties from the vector bundle case.
In particular \eqref{sommar} and \eqref{sommarbad} hold.
Also
\eqref{veckoslut} holds and, using that
\begin{equation*}
  \ln \big(c(E, D)\big) = \ln \left (\prod_{k = 0}^N c(E_k, D_k)^{(-1)^k}
  \right) = \sum_{k = 0}^N (-1)^k \ln \big(c(E_k, D_k)\big),
\end{equation*}
we get that
  \begin{equation}\label{vitsippa}
       \ch_\ell(E,D) = \sum_{k = 0}^N (-1)^k \ch_\ell(E_k,D_k).
    \end{equation}
    In particular, $\ch_\ell(E, D)$ represents $\ch_\ell(E)$.

    \smallskip

    Let $\Theta_k$ denote the curvature matrix of $D_k$ (in some local
    trivialization)
and define\footnote{To be consistent with \eqref{tetrapak} we should
  have a factor $(i/2\pi)^\ell$ in the definition of
  $p_\ell(E,D)$, \emph{cf.}\  ~\eqref{eq:chernCharacterAlternative}. However,
  the normalization \eqref{eq:pldef} is more convenient to work
with.}
\begin{equation}\label{eq:pldef}
    p_\ell(E,D) = \sum_{k = 0}^N (-1)^k \tr \Theta_k^\ell.
  \end{equation}
In view of \eqref{lovande} and \eqref{vitsippa}, for $\ell\geq 1$,
\begin{equation} \label{eq:chernCharacterAlternative}
  \ch_\ell(E,D) =
 \frac{i^\ell}{(2\pi)^\ell \ell!}
  p_\ell(E,D).
\end{equation}

Assume that $D = (D_k)$ is a $(1,0)$-connection. Let $()_{(q,r)}$ denote the part of bidegree $(q,r)$
of a form. Since the curvature matrices $\Theta_k$ (in local trivializations) consist of forms of
bidegree $(2,0)$ and $(1,1)$, it follows that
\begin{equation} \label{eq:pl11}
    p_{(\ell,\ell)}(E,D) = \sum_{k = 0}^N (-1)^k \tr (\Theta_k)_{(1,1)}^\ell
\end{equation}
and by \eqref{eq:chernCharacterAlternative} that
\begin{equation} \label{eq:chernCharacterAlternativell}
 \ch_{(\ell,\ell)}(E,D) = 
  \frac{i^\ell}{(2\pi)^\ell \ell!}
  p_{(\ell,\ell)}(E,D).
\end{equation}

\section{Connections compatible with a complex}\label{taget}

Assume that $(E, \varphi)$ is a complex of vector bundles of the form
\begin{equation} \label{eq:augmentedEcomplex}
     0 \to E_N \xrightarrow[]{\varphi_N}  \cdots\xrightarrow[]{\varphi_1}  E_0\xrightarrow[]{\varphi_0} E_{-1} \to 0.
\end{equation}
Moreover assume that each $E_k$ is equipped with a connection
$D_k$. Then, following \cite{BB}, we say that the connection $D = (D_{-1}, \ldots, D_N)$
on $(E,\varphi)$
is \emph{compatible} with $(E, \varphi)$ if
\begin{equation} \label{eq:compatible}
    D_{k-1} \circ \varphi_k = -\varphi_k \circ D_k
  \end{equation}
  for $k = 0, \ldots, N$.
In terms of the induced connection $D = D_{\End}$ on $\End E$, this can
succinctly be written as $D\varphi_k = 0$.

Note that in contrast to above, \eqref{eq:augmentedEcomplex} starts at level $-1$.
The typical situation we consider is when we start with a complex
\eqref{turkos} that is pointwise exact outside an analytic variety $Z$ and
then restrict to $X \setminus Z$; then
$E_{-1} = 0$.

\begin{remark}\label{teater}
By \cite[Lemma 4.17]{BB}, given a complex $(E, \varphi)$ of vector
bundles of the form \eqref{eq:augmentedEcomplex}
one can always extend a given connection $D_{-1}$ on $E_{-1}$
to a connection $D = (D_k)$ that is compatible with
$(E, \varphi)$ where it is pointwise exact.
In fact, Lemma~\ref{lma:compatible} below gives an explicit formula for such
a connection, see Remark~\ref{national}.
\end{remark}

\begin{remark}
In \cite{BB}, the condition of being compatible is stated without the
minus sign in \eqref{eq:compatible}; our condition on $D$ is actually the same,
but we need to introduce the minus sign since we use the conventions of the superstructure.
Indeed, if $\xi$ is a section of $E_k$ of form-degree $0$, then
$D_{k-1} \circ \varphi_k \xi$ is defined in the same way with or without the
superstructure, while the action of $\varphi_k$ on $D_k \xi$ changes sign depending on whether
the superstructure is used or not since $D_k\xi$ has form-degree $1$,
\emph{cf.}\ \eqref{eq:superaction}.
\end{remark}

Compatible connections satisfy the following Whitney formula,
\cite[Lemma~4.22]{BB},  \emph{cf.}\ Section \ref{kalla}.
\begin{lemma} \label{lma:BBvanishing}
    Assume that $(E, \varphi)$ is an exact complex of vector
    bundles of the form \eqref{eq:augmentedEcomplex} that is
    equipped with a connection $D = (D_k)$ that is compatible with $(E, \varphi)$.
      Then
    \begin{equation*}
        c(E_{-1},D_{-1}) = \prod_{k = 0}^N c(E_k,D_k)^{(-1)^k}.
    \end{equation*}
\end{lemma}

\subsection{The connection $\widehat D^\epsilon$}
We will consider a specific situation and choice of compatible
connection.
As in previous sections, let $(E, \varphi)$ be a complex of
vector bundles of the form \eqref{turkos} that is pointwise exact outside the
analytic set $Z$.
Moreover, let $\chi$ be a smooth approximand of $\chi_{[1,\infty)}$, let $F$ be a (generically nonvanishing) section of a vector bundle such
that $Z \subseteq \{F = 0\}$, and let
$\chi_\epsilon = \chi(|F|^2/\epsilon)$. Then $\chi_\epsilon\equiv 0$ in
a neighborhood of $Z$.
Consider now a fixed choice of connection $D = (D_k)$ on $(E, \varphi)$,
and for $\epsilon > 0$, define a new connection
$\widehat D^\epsilon = (\widehat D^\epsilon_k)$ on $(E,\varphi)$
through
\begin{equation}\label{eq:compatibleConnection}
    \widehat D_k^\epsilon = -\chi_\epsilon \sigma_k D\varphi_k + D_k.
\end{equation}
Note that if $D$ is a $(1,0)$-connection, then so is $\widehat
D^\epsilon$.

\begin{lemma} \label{lma:compatible}
  The connection $\widehat{D}^\epsilon$ is compatible with
  $(E,\varphi)$ where $\chi_\epsilon \equiv 1$. 
\end{lemma}

\begin{proof}
    Using \eqref{DEnd}, \eqref{exakt} and \eqref{identitet}
    we obtain that
    \begin{align*}
        \widehat D^\epsilon \varphi_k &= \widehat D_{k-1}^\epsilon
        \circ \varphi_k + \varphi_k \circ \widehat D_k^\epsilon \\
        &=\chi_\epsilon \left(-\sigma_{k-1} D\varphi_{k-1} \varphi_k - \varphi_k \sigma_k D\varphi_k \right) + D_{k-1} \circ \varphi_k + \varphi_k \circ D_k \\
       &= -\chi_\epsilon (\sigma_{k-1} \varphi_{k-1} + \varphi_k \sigma_k) D\varphi_k + D_{k-1} \circ \varphi_k + \varphi_k \circ D_k\\
       &= (1-\chi_\epsilon) D \varphi_k.
      \end{align*}
      In particular, $\widehat D^\epsilon$ is compatible
    with the complex where $\chi_\epsilon \equiv 1$.
\end{proof}

\begin{remark}\label{national}
Assume that $(E, \varphi)$ is a pointwise exact complex of vector
bundles equipped with some connection $D = (D_k)$. Then, as in the proof
above, it follows that the connection $\widetilde D$ defined by 
\[
  \widetilde D_k = -\sigma_k D\varphi_k + D_k 
\]
is compatible with $(E, \varphi)$. Moreover, $\widetilde D_{-1} = D_{-1}$,
\emph{cf.}\ Remark \ref{teater}.
\end{remark}

Assume that $\theta_k$ is a connection matrix for $D_k$ in a local
trivialization, \textit{i.e.}
$D_k\alpha = d\alpha + \theta_k \wedge \alpha$.
Then the connection matrix for $\widehat{D}_k^\epsilon$ is
\[
  \hat{\theta}^\epsilon_k = -\chi_\epsilon \sigma_k D\varphi_k + 
  \theta_k
\]
and thus the curvature matrix of $\widehat D_k^\epsilon$ equals
\begin{equation}\label{trubbel}
    \widehat{\Theta}_k^\epsilon = d\hat{\theta}^\epsilon_k + 
    (\hat{\theta}_k^\epsilon)^2
    =
- d(\chi_\epsilon \sigma_k D\varphi_k) + \chi_\epsilon^2 \sigma_k D\varphi_k \sigma_kD\varphi_k -\chi_\epsilon(\theta_k \sigma_k
    D\varphi_k + \sigma_k D\varphi_k \theta_k) + \Theta_k,
 \end{equation}
where $\Theta_k$ is the curvature matrix of $D_k$.

 \section{The Chern current $c^{\Res}(E, D)$}
\label{gula}

In this section we prove that the limits as $\epsilon \to 0$ of
products of Chern forms $c_\ell(E, \widehat
D^\epsilon)$, where $\widehat D^\epsilon$ is the connection from the
previous section, give the desired currents in \eqref{eq:chernCurrent}.
More generally, we prove the following
generalization of Theorem~\ref{thm:existence}.

\begin{theorem} \label{thm:existence2}
Assume that $(E,\varphi)$ is a complex of Hermitian vector bundles of the form
\eqref{turkos} that is pointwise exact
outside a subvariety $Z$ of positive codimension. Moreover assume that
$D = (D_k)$ is a connection on $(E,\varphi)$ and let $\widehat
D^\epsilon$ be the connection defined by \eqref{eq:compatibleConnection}.
    Then, for $\ell_1, \ldots, \ell_m\in \N_{>0}$, 
   \begin{equation} \label{eq:existence3}
        c_{\ell_1}^{\Res}(E,D) \wedge \dots \wedge c_{\ell_m}^{\Res}(E,D) = 
        \lim_{\epsilon \to 0} c_{\ell_1}(E,\widehat{D}^\epsilon) \wedge
        \dots \wedge c_{\ell_m}(E,\widehat{D}^\epsilon),
    \end{equation}
where the right side is defined by \eqref{syster}, is a well-defined
closed
pseudomeromorphic current that is independent of
the choice of $\chi_\epsilon$, has support on $Z$, and
represents $c_{\ell_1}(E)\wedge\cdots\wedge c_{\ell_m}(E)$.
\end{theorem}

Theorem~\ref{thm:existence} is an immediate consequence of Theorem~\ref{thm:existence2}.

\begin{proof}
Let
    \begin{equation} \label{eq:Mepsilon}
        M_\epsilon = c_{\ell_1}(E,\widehat{D}^\epsilon) \wedge
        \dots \wedge c_{\ell_m}(E,\widehat{D}^\epsilon).
      \end{equation}
      We first prove that $\lim_{\epsilon\to 0}M_\epsilon$ exists and
      is a pseudomeromorphic current.
This is a local statement and we may therefore work in a local trivialization
where $D_k$ is determined by the connection matrix $\theta_k$.
By \eqref{trubbel},
 $\widehat\Theta_k^\epsilon$ is a (form-valued) matrix of the form
     \begin{equation*}
       \widehat{\Theta}_k^\epsilon = 
       \alpha_k + 
       \chi_\epsilon \beta_k' + \chi_\epsilon^2 \beta_k'' + 
       d\chi_\epsilon \wedge \beta_k''',
      \end{equation*}
  where $\alpha_k = \Theta_k$ is smooth and $\beta_k'$, $\beta_k''$ and $\beta_k'''$ are polynomials in
    $\sigma_k$, $D\varphi_k$, $\theta_k$ and exterior derivatives of
    such factors. In particular $\alpha_k$, $\beta_k'$, $\beta_k''$,
    and $\beta_k'''$ are
    independent of $\epsilon$.

Since $M_\epsilon$ is a polynomial in the entries of
    $\widehat{\Theta}_0^\epsilon,\dots,\widehat{\Theta}_N^\epsilon$,
    see Section \ref{kalla},
   we can write
    \begin{equation*} 
        M_\epsilon = A + \sum_{j \geq 1} \chi_\epsilon^j B'_j + \sum_{j \geq 1} \chi_\epsilon^{j-1} d\chi_\epsilon\wedge B''_j,
    \end{equation*}
    where $A$, $B'_j$, and $B''_j$ are independent of
    $\epsilon$, $A$ is smooth, and $B'_j$ and $B''_j$ are
    polynomials in entries of  $\sigma_k$, $D\varphi_k$, $\theta_k$
    and exterior derivatives of such factors.
 Let
    $\hat\chi_\epsilon = \hat\chi(|F|^2/\epsilon)$, where $\hat
    \chi = \chi^j\sim \chi_{[1,\infty)}$, \emph{cf.}\ Section
    ~\ref{stromprelim}.
    Then by Lemma
    ~\ref{lma:PMexistence}, the limits of
    \begin{equation*}
      \chi^j_\epsilon B_j' = \hat \chi_\epsilon B_j'
      \qquad\text{and}\qquad
      \chi^{j-1}_\epsilon d\chi_\epsilon \wedge B_j'' = d\hat
      \chi_\epsilon \wedge B_j''/j
    \end{equation*}
as $\epsilon \to 0$ exist and are pseudomeromorphic
currents that are independent of $\chi_\epsilon$. It follows that the limit \eqref{eq:existence3} exists and
is a pseudomeromorphic current that is independent of
$\chi_\epsilon$.

By Lemma~\ref{lma:compatible}, $\widehat D^\epsilon$ is compatible with
$(E, \varphi)$ where $\chi_\epsilon\equiv 1$ and therefore, by Lemma~\ref{lma:BBvanishing}, $c(E, \widehat D^\epsilon) = 0$ there. 
It follows that $M_\epsilon$ has support where $\chi_\epsilon\not\equiv 1$.
Note that the $\sigma_k$ are smooth outside of $Z$. By Lemma~\ref{lma:PMexistence}, 
the limit \eqref{eq:existence3} is independent of the choice of $\chi_\epsilon$.
In particular, we may assume that the section $F$ defining $\chi_\epsilon = \chi(|F|^2/\epsilon)$
is locally defined such that $\{ F = 0 \} = Z$. It then follows that the limit \eqref{eq:existence3} has support on $Z$.
That \eqref{eq:existence3} represents
$c_{\ell_1}(E)\wedge\cdots\wedge c_{\ell_m}(E)$ follows by Poincar\'e duality,
since the forms on the right hand side of \eqref{eq:existence3}
represent this class for all $\epsilon > 0$.
Also \eqref{eq:existence3} is closed since the forms on the right hand
side are for all $\epsilon>0$.
\end{proof}

\begin{remark}\label{locket}
Assume that $D = (D_k)$ in Theorem \ref{thm:existence2} is a
$(1,0)$-connection. Then
 $\widehat{\Theta}_k^\epsilon$ only has components of bidegree $(2,0)$
 and $(1,1)$, \emph{cf.}\ \eqref{trubbel}. It follows that \eqref{eq:Mepsilon}
 and consequently \eqref{eq:existence3} consist
  of components of bidegree $(\ell + q,\ell-q)$ with $q\geq 0$,
  where $\ell = \ell_1 + \cdots + \ell_m$.
   \end{remark}

\section{An explicit description of Chern currents of low degrees}\label{chernsec}

In this section we study the Chern current $c^{\Res}(E, D)$ of a complex
$(E, \varphi)$ that is
equipped with a $(1,0)$-connection $D$.
Our main result is the following generalization of
Theorem~\ref{thm:main} that
is an explicit description of $c_p^{\Res}(E, D)$ in terms of the
residue current $R$ associated with $(E, \varphi)$.

\begin{theorem} \label{thm:mainGeneral}
    Assume that $(E,\varphi)$ is a complex of Hermitian vector bundles
    of the
    form \eqref{turkos} that is pointwise exact outside a subvariety
    $Z$ of codimension $p$, and let $R$ be the corresponding residue
    current.
    Moreover, assume that $D = (D_k)$ is a $(1,0)$-connection on
    $(E,\varphi)$ and let $c^{\Res}(E, D)$ be the corresponding Chern
    current.
    Then
    \begin{equation}\label{glashus}
        c_p^{\Res}(E,D) = \frac{(-1)^{p-1}}{(2\pi i)^p p} \sum_{k = 0}^{N-p} (-1)^k \tr (D\varphi_{k + 1} \cdots D\varphi_{k + p} R^{k}_{k + p}).
      \end{equation}
Moreover
  \begin{equation}\label{sommardag}
    c_\ell^{\Res}(E,D) = 0 \quad \text{ for } 0 < \ell < p
  \end{equation}
  and
  \begin{equation}\label{badhus}
    c_{\ell_1}^{\Res}(E,D) \wedge \dots \wedge c_{\ell_m}^{\Res}(E,D) = 
    0  \quad \text{ for } m \geq 2 \text{ and } 0  < \ell_1 + \dots + \ell_m \leq p.
    \end{equation}
\end{theorem}

In fact, Theorem \ref{thm:mainGeneral} follows from the following
formulation in terms of the Chern character (forms).
For $(E, \varphi)$ and $D$ as in the theorem and for $\ell_1,\dots,\ell_m \geq 1$ we let
\begin{equation}\label{husvagn}
  \ch^{\Res}_{\ell_1}(E, D) \wedge \dots \wedge \ch^{\Res}_{\ell_m}(E, D):= 
  \lim_{\epsilon\to 0} \ch_{\ell_1}(E, \widehat D^\epsilon) \wedge \dots\wedge \ch_{\ell_m}(E, \widehat D^\epsilon),
\end{equation}
where $\widehat D^\epsilon$ is the connection defined by
\eqref{eq:compatibleConnection}.
By Theorem \ref{thm:existence2} this is a well-defined current with
support on $Z$ that represents $\ch_{\ell_1}(E)\wedge\dots\wedge \ch_{\ell_m}(E)$.

\begin{theorem}\label{saltkar}
Assume that $(E, \varphi)$, $D$, $R$, and $p$ are as in Theorem
\ref{thm:mainGeneral}. For $\ell\geq 1$, let $\ch_\ell^{\Res}(E, D)$ be the corresponding
Chern character current \eqref{husvagn}. Then
\begin{equation}\label{glashus2}
        \ch_p^{\Res}(E,D) = \frac{1}{(2\pi i)^p p!} \sum_{k = 0}^{N-p} (-1)^k \tr (D\varphi_{k + 1} \cdots D\varphi_{k + p} R^{k}_{k + p}).
      \end{equation}
Moreover
  \begin{equation}\label{glashus3}
    \ch_\ell^{\Res}(E,D) = 0 \quad \text{ for }  \ell < p
  \end{equation}
  and
  \begin{equation}\label{glashus4}
    \ch_{\ell_1}^{\Res}(E,D) \wedge \dots \wedge \ch_{\ell_m}^{\Res}(E,D) = 
    0 \quad \text{ for } m \geq 2 \text{ and }  \ell_1 + \dots + \ell_m \leq p.
    \end{equation}
 \end{theorem}

\begin{proof}[Proof of Theorem~\ref{thm:mainGeneral}]
Since $\widehat Q_\ell$ in \eqref{sommarbad} is a
polynomial of weighted degree $\ell$ we get that
\begin{equation}\label{kvavt}
\widehat Q_\ell\big(\ch_1^{\Res}(E, D), \ldots, \ch_{\ell-1}^{\Res}(E,
D)\big)
\end{equation}
is a sum of terms
\begin{equation}\label{verkligen}
  \ch_{\lambda_1}^{\Res}(E,D) \wedge \dots \wedge
  \ch_{\lambda_s}^{\Res}(E,D),  \text{ where } s\geq 2 \text{ and } \lambda_1 + \cdots + \lambda_s = \ell.
  \end{equation}
  Thus \eqref{kvavt}
vanishes by \eqref{glashus4} for $\ell\leq p$. Now \eqref{glashus} and
\eqref{sommardag} follow by \eqref{sommarbad},
\eqref{glashus2}, and \eqref{glashus3}.
Also, the left hand side of \eqref{badhus} is a sum of terms of the
form \eqref{verkligen} and thus it vanishes.
\end{proof}

\begin{remark}\label{badkar}
  Taking Theorem \ref{thm:mainGeneral} for granted, by similar
  arguments as in the proof above, using
  \eqref{sommar}, we get Theorem \ref{saltkar}. Thus
  Theorems \ref{thm:mainGeneral} and \ref{saltkar} are equivalent.
  \end{remark}

Recall from Section~\ref{ssect:AW1} that if
  $(E,\varphi)$ is a locally free resolution of a sheaf $\mathcal F$
of codimension $p$,
then $R^\ell_k = 0$ for
$\ell>0$. Thus the only nonvanishing residue current in
\eqref{glashus} is $R_p = R_p^0$, and hence
Theorem~\ref{thm:main} follows.
It may be noted that our proof of Theorem~\ref{thm:mainGeneral} does not become simpler
in the situation of Theorem~\ref{thm:main}.

\smallskip

To organize the proof of Theorem~\ref{saltkar} we will introduce
a certain class $O_{Z,\ell,\epsilon}$ of forms
depending on $\epsilon>0$ that in the limit are pseudomeromorphic
currents with support on $Z$ that vanish if $\ell\leq \codim Z$.
Throughout this section, let $(E,\varphi)$ be fixed as the complex
from Theorems~\ref{thm:mainGeneral} and~\ref{saltkar}
and let $\sigma_k$ be the minimal inverse of $\varphi_k$ as in Section~\ref{ssect:AW1}. Let $\mathcal{E}_{q,\epsilon}$ denote smooth forms of bidegree $(*,q)$
that can be written as polynomials in $\chi_\epsilon$, $\dbar\chi_\epsilon$,
entries of $\sigma_k$, $\partial\sigma_k$ or $\dbar\sigma_k$ in some
local trivialization for $k = 1,\dots,N$, and smooth forms independent of $\epsilon$.
Here
$\chi_\epsilon = \chi(|F|^2/\epsilon)$, where $\chi$ is a smooth
approximand of $\chi_{[1,\infty)}$ and $F$ is a generically non-vanishing section of a
holomorphic vector bundle such that $Z = \{F = 0\}$.
We say that $\psi_\epsilon\in \mathcal E_\epsilon := \oplus \mathcal{E}_{q,\epsilon}$ is in $O_{Z,\ell,\epsilon}$
if $\psi_\epsilon$ is a sum of terms of the form
$a\wedge b_{\epsilon}$, where $a$ is a smooth
form that is independent of
$\epsilon$, and $b_{\epsilon}$ is in $\mathcal E_{q, \epsilon}$, where
$q<\ell$, and vanishes where
$\chi_\epsilon \equiv 1$.
In particular, if $\psi_\epsilon\in \mathcal E_{q,\epsilon}$ vanishes
where $\chi_\epsilon \equiv 1$, then $\psi_\epsilon\in O_{Z,\ell,
  \epsilon}$ for any $\ell>q$.
Note that
\begin{equation} \label{eq:Oideal}
    \mathcal{E}_{q,\epsilon} \wedge O_{Z,\ell,\epsilon} \subseteq O_{Z,\ell + q,\epsilon}.
\end{equation}

\begin{lemma} \label{lma:OdimensionPrinciple}
    Assume that $Z$ has codimension $p$, and let $\psi_\epsilon$ be a form in $O_{Z,\ell,\epsilon}$
    with $\ell \leq p$.
    Then
       $ \lim_{\epsilon \to 0} \psi_\epsilon = 0.$
\end{lemma}

\begin{proof}
    Consider a term $a \wedge b_\epsilon$ of $\psi_\epsilon$
    as above. Then $b_\epsilon \in \mathcal E_{q,\epsilon}$, where $q<\ell\leq p$.
By Lemma~\ref{lma:PMexistence}, the limit $b := \lim_{\epsilon \to 0}
b_\epsilon$ exists and
is a pseudomeromorphic current of bidegree $(*,q)$. Since  $b_\epsilon
    \equiv 0$ where $\chi_\epsilon \equiv 1$, $b$ has support on $Z$
    and thus $b = 0$ by the dimension principle, see Section~\ref{hemma}.
    Since $a$ is smooth and independent of $\epsilon$, it follows that
    $\lim_{\epsilon \to 0} (a\wedge b_\epsilon) = a \wedge b = 0$.
\end{proof}

Throughout this section, let $D = (D_k)$ be a $(1,0)$-connection on $(E,
\varphi)$, let $\chi\sim
\chi_{[1,\infty)}$, let $\chi_\epsilon$ and  $\widehat D^\epsilon$ be defined
as in Section~\ref{taget}, and let
$c(E, \widehat D^\epsilon) = \sum c_\ell(E,
\widehat D^\epsilon)$ be the corresponding Chern form defined by
\eqref{syster}.
Since the limits in Theorem \ref{thm:existence2} are independent of the choice of $\chi_\epsilon$,
and the results in this section are local statements, we may assume locally that the section $F$ in the definition 
of $\chi_\epsilon=\chi(|F|^2/\epsilon)$ is such that $\{ F = 0 \} = Z$.

\begin{lemma}\label{struntsak}
    For $\ell\geq 1$ and $\epsilon >0$, we have
     \begin{equation}\label{regna}
  \ch_\ell(E,
\widehat D^\epsilon) = 
        \frac{1}{(2\pi i)^\ell \ell!} \dbar\chi_\epsilon^\ell \wedge \sum_{k = 1}^N (-1)^k
        \tr \big(\sigma_k D\varphi_k(\dbar\sigma_k
        D\varphi_k)^{\ell-1}\big) + 
       O_{Z,\ell,\epsilon}.
    \end{equation}
\end{lemma}

\begin{proof}
  We may work in a local trivialization; let $\widehat
  \Theta_k^\epsilon$ be the curvature matrix of $\widehat
  D^\epsilon_k$.
  By Remark \ref{locket}, since the $D_k$ are $(1,0)$-connections, $\ch_\ell(E, \widehat D^\epsilon)$ consists of
  components of bidegree $(\ell + q,\ell-q)$ with $q \geq 0$.
From the proof of Theorem~\ref{thm:existence2} it follows that $c(E,\widehat D^\epsilon)$ is
in $\mathcal E_\epsilon$ and vanishes where $\chi_\epsilon\equiv 1$, and consequently,
the same holds for $\ch(E,\widehat D^\epsilon)$.
It follows that
$\ch_{(\ell + q,
  \ell-q)}(E, \widehat D^\epsilon)\in O_{Z,\ell,\epsilon}$ for $q>0$,
so
  \begin{equation} \label{eq:chlll}
      \ch_\ell(E,\widehat D^\epsilon) = \ch_{(\ell,\ell)}(E, \widehat
      D^\epsilon) + O_{Z,\ell,\epsilon}.
  \end{equation}
Since $\widehat D^\epsilon$ is a $(1,0)$-connection, by
\eqref{eq:chernCharacterAlternativell},
\begin{equation}\label{husdjur}
  \ch_{(\ell,\ell)}(E, \widehat
  D^\epsilon) = 
  \frac{i^\ell}{(2\pi)^\ell \ell!} p_{(\ell, \ell)}(E, \widehat D^\epsilon),
\end{equation}
where $p_{(\ell, \ell)}$ is given by \eqref{eq:pl11}.
To prove the lemma it thus suffices to show that
\begin{equation}\label{eq:pl}
         p_{(\ell, \ell)}(E, \widehat D^\epsilon) = (-1)^\ell \dbar\chi_\epsilon^\ell \wedge \sum_{k = 1}^N
        (-1)^k \tr \big(\sigma_k D\varphi_k (\dbar\sigma_k
        D\varphi_k)^{\ell-1}\big) + O_{Z,\ell,\epsilon}.
      \end{equation}

\smallskip

    To prove \eqref{eq:pl}, first note in view of \eqref{trubbel} that since $D$ is a $(1,0)$-connection,
\begin{equation}\label{snabel}
  (\widehat{\Theta}_k^\epsilon)_{(1,1)} = -\dbar(\chi_\epsilon
  \sigma_k D\varphi_k) + (\Theta_k)_{(1,1)},
\end{equation}
where $\Theta_k$ is the curvature matrix of $D_k$.
We make the following decomposition:
\begin{equation}
  \label{fabel}
  \begin{split}
   -\dbar(\chi_\epsilon \sigma_k D\varphi_k) + (\Theta_k)_{(1,1)}
    & = -\dbar\chi_\epsilon \wedge \sigma_k D\varphi_k -
    \chi_\epsilon \big(\dbar(\sigma_k D\varphi_k)
   - (\Theta_k)_{(1,1)} \big) + (1-\chi_\epsilon)(\Theta_k)_{(1,1)} \\
    & = : \alpha_k + \beta_k + \gamma_k.
  \end{split}
\end{equation}
Let us consider
\begin{equation*}
  \tr (\widehat{\Theta}_k^\epsilon)^\ell_{(1,1)} = \tr
  (\alpha_k + \beta_k + \gamma_k)^\ell
\end{equation*}
and expand the product.  Note that $\gamma_k\in O_{Z,1,\epsilon}$, and
thus by \eqref{eq:Oideal} all terms with a factor $\gamma_k$ are in
$O_{Z, \ell, \epsilon}$.  Next, note that since
$(\dbar\chi_\epsilon)^2 = 0$, $\alpha_k^2 = 0$, and since $\alpha_k$
and $\beta_k$ have total degree $4$ and endomorphism degree $2$, all
terms containing one $\alpha_k$ and the remaining $\ell-1$ factors
being $\beta_k$ are all equal to 
$\tr (\alpha_k \wedge \beta_k^{\ell-1})$ by \eqref{spar}.
To conclude,
\begin{equation} \label{eq:trThetakl}
  \tr (\widehat{\Theta}_k^\epsilon)^\ell_{(1,1)}
  = \ell \tr (\alpha_k \wedge\beta_k^{\ell-1})
  + \tr \beta_k^\ell + O_{Z,\ell,\epsilon}.
\end{equation}
We have that
\begin{align*}
  \ell\tr (\alpha_k \wedge \beta_k^{\ell-1})
  &= (-1)^\ell \dbar\chi_\epsilon^\ell \wedge
    \tr \big(\sigma_k D\varphi_k
    \big(
    \dbar(\sigma_k D\varphi_k)-(\Theta_k)_{(1,1)}
    \big)^{\ell-1}\big) \\
  &= (-1)^\ell \dbar\chi_\epsilon^\ell \wedge \tr \big(\sigma_k D\varphi_k
  \big(\dbar(\sigma_k D\varphi_k)\big) ^{\ell-1} \big) + O_{Z,\ell,\epsilon},
\end{align*}
since
$\ell\chi_\epsilon^{\ell-1}\dbar\chi_\epsilon = \dbar\chi_\epsilon^\ell$,
\emph{cf.}\ \eqref{information},
and in the middle expression all terms having a factor  $(\Theta_k)_{(1,1)}$
also contain a factor $\dbar\chi_\epsilon^\ell$, and thus are in $O_{Z,\ell,\epsilon}$.
Moreover, by \eqref{eq:dbarLeibniz} and \eqref{eq:dbarD},
\begin{equation*}
  \dbar(\sigma_k D\varphi_k) = 
  \dbar\sigma_k D\varphi_k -\sigma_k \dbar (D\varphi_k) = 
  \dbar\sigma_k D\varphi_k - \sigma_k
  (\Theta_{k-1})_{(1,1)}\varphi_k + \sigma_k \varphi_k (\Theta_k)_{(1,1)},
\end{equation*}
and hence
\begin{equation}\label{eq:tralphaBetal}
  \ell\tr (\alpha_k \wedge \beta_k^{\ell-1}) = 
  (-1)^\ell \dbar\chi_\epsilon^\ell \wedge \tr \big(\sigma_k D\varphi_k
  \big(\dbar\sigma_k D\varphi_k\big) ^{\ell-1} \big) + 
  O_{Z,\ell,\epsilon},
\end{equation}
since all terms containing a factor $(\Theta_k)_{(1,1)}$ or $(\Theta_{k-1})_{(1,1)}$ also contain a factor $\dbar\chi_\epsilon^\ell$.
Note that $\alpha_0 = 0$ since $\varphi_0 = 0$. It thus follows from \eqref{eq:trThetakl}
and \eqref{eq:tralphaBetal} that
\begin{align*}
  p_{(\ell, \ell)}(E, \widehat D^\epsilon)
  &= \sum_{k = 0}^N (-1)^k
    \tr (\widehat\Theta_k^\epsilon)_{(1,1)}^\ell \\
  &= (-1)^\ell  \dbar\chi_\epsilon^\ell \wedge
    \sum_{k = 1}^N (-1)^k \tr \big(\sigma_k D\varphi_k
    \big(\dbar\sigma_k D\varphi_k\big) ^{\ell-1} \big) + 
    \sum_{k = 0}^N (-1)^k \tr \beta_k^\ell + O_{Z,\ell,\epsilon}.
\end{align*}

Thus, to prove \eqref{eq:pl} it suffices to show that 
$\sum_{k = 0}^N (-1)^k \tr \beta_k^\ell$ vanishes for $\ell\geq 1$.
Outside $Z$, let $\widetilde D$ be the connection on $(E, \varphi)$
defined by
\begin{equation}\label{bibliotek}
  \widetilde{D}_k := -\sigma_kD\varphi_k + D_k
\end{equation}
and let $\tilde c = c(E, \widetilde D)$ be the corresponding Chern
form defined by \eqref{syster}.  If follows from Lemma~\ref{lma:compatible} that $\widetilde D$ is compatible with 
$(E, \varphi)$ and thus by Lemma~\ref{lma:BBvanishing}, 
${\tilde c}_j$ vanishes for $j\geq 1$.  For $\ell\geq 1$, let 
$\tilde p_\ell := p_{(\ell,\ell)}(E,\widetilde D)$, where
$p_{(\ell,\ell)}$ is given by \eqref{eq:pl11}.  By
\eqref{eq:chernCharacterAlternativell} and \eqref{sommar}, 
$\tilde p_\ell$ is a polynomial in 
$\tilde c_{(1,1)},\dots,\tilde c_{(\ell,\ell)}$.  Since
$\tilde{c}_{(j,j)}$ vanishes for any $j \geq 1$, 
$\tilde p_\ell = 0$.  
Note that $\beta_k = \chi_\epsilon (\widetilde{\Theta}_k)_{(1,1)}$, where
$\widetilde{\Theta}_k$ is the curvature matrix corresponding to
$\widetilde D_k$.  Thus
\begin{equation*}
  \sum_{k = 0}^N (-1)^k \tr \beta_k^\ell
  = \chi_\epsilon^\ell \sum_{k = 0}^N (-1)^k \tr
  (\widetilde \Theta_k)_{(1,1)}^\ell
  = \chi_\epsilon^\ell \tilde{p}_\ell = 0
\end{equation*}
for $\epsilon >0$.  This concludes the proof of \eqref{eq:pl}.
\end{proof}

\begin{lemma}\label{klunga}
For $\ell\geq 1$ and $\epsilon >0$, we have
\begin{multline}
  \label{sjunga}
  \dbar \chi_\epsilon \wedge \sum_{k = 1}^N
  (-1)^{k}\tr \big(\sigma_k D\varphi_k (\dbar\sigma_k
  D\varphi_k)^{\ell-1} \big) = \\
  (-1)^\ell \sum_{k = 0}^{N-\ell} (-1)^k
  \dbar \chi_\epsilon \wedge \tr
  \big(\sigma_{k + \ell} \dbar\sigma_{k + \ell-1}
  \cdots \dbar\sigma_{k + 1} D\varphi_{k + 1}
  \cdots D\varphi_{k + \ell} \big) + O_{Z,\ell,\epsilon}
  + \dbar O_{Z,\ell,\epsilon}.
\end{multline}
If $\ell > N$, the sum on the right hand side should be interpreted as $0$.
\end{lemma}
Here $\dbar O_{Z,\ell,\epsilon}$ means forms of the form $\dbar
\psi_\epsilon$, where $\psi_\epsilon\in O_{Z,\ell,\epsilon}$.

\begin{proof}
    For $\ell = 1$ the sums differ only by a shift in the indices, so we may assume $\ell \geq 2$.
    For fixed $k\in \Z$ and $m,r,s \geq 0$, let
    \begin{equation*}
        \rho_{k,m}^{r,s} = \dbar \chi_\epsilon \wedge \tr
        \big(\sigma_{k + m + 1} \dbar \sigma_{k + m}\cdots \dbar\sigma_{k + 1}
        (\dbar \sigma_{k} D\varphi_{k})^r
        D\dbar \sigma_{k}
        (D\varphi_{k}\dbar\sigma_{k})^{s}
        D \varphi_{k} \cdots D\varphi_{k + m} \varphi_{k + m + 1} \big).
    \end{equation*}
    If $m = 0$, then the factor
    $\dbar\sigma_{k + m}\cdots\dbar\sigma_{k + 1}$ should be
    interpreted as $1$. Moreover since $(E,\varphi)$ starts at level
    $0$ and ends at level $N$, we interpret $\varphi_j$ and $\sigma_j$ as $0$ if
   $j>N$ or $j<1$, and consequently we interpret
    $\rho_{k,m}^{r,s}$ as $0$ if $k + m\geq N$ or $k\leq 0$.

    We claim that
    \begin{multline}
        \label{unge}
        \dbar \chi_\epsilon \wedge \tr \big(\sigma_{k + m} \dbar \sigma_{k + m-1}\cdots \dbar\sigma_k
        (\dbar \sigma_{k-1} D\varphi_{k-1})^r (D\varphi_{k}\dbar\sigma_k)^{s + 1}
        D\varphi_k\cdots D\varphi_{k + m} \big)
       = \\
        \dbar \chi_\epsilon \wedge \tr \big(\sigma_{k + m} \dbar \sigma_{k + m-1}\cdots \dbar\sigma_k
        (\dbar \sigma_{k-1} D\varphi_{k-1})^{r + 1} (D\varphi_{k}\dbar\sigma_k)^{s}
        D\varphi_k\cdots D\varphi_{k + m} \big) \\ + \rho_{k-1,m}^{r,s} + \rho_{k,m}^{r,s} + O_{Z,r + s + m + 2,\epsilon} + \dbar O_{Z,r + s + m + 2,\epsilon}.
    \end{multline}

    Let us take \eqref{unge} for granted and let
    $\rho_{k,m}^t = \sum_{r = 0}^t\rho_{k,m}^{r,t-r}$. Then by inductively
    applying \eqref{unge} to $r = 0,\ldots, t$ with $s = t-r$, we get
     \begin{multline}
        \label{algebra}
        \dbar \chi_\epsilon \wedge \tr \big(
        \sigma_{k + m} \dbar \sigma_{k + m-1}\cdots \dbar \sigma_k
        (D\varphi_{k}\dbar\sigma_k)^{t + 1}
        D\varphi_k\cdots D\varphi_{k + m} \big)
       = \\
        \dbar \chi_\epsilon \wedge \tr \big(
        \sigma_{k + m} \dbar\sigma_{k + m-1}\cdots \dbar \sigma_{k-1}
        (D\varphi_{k-1}\dbar\sigma_{k-1})^{t}
        D\varphi_{k-1}\cdots D\varphi_{k + m}\big)
        \\ + \rho_{k-1,m}^{t} + \rho_{k,m}^{t} + O_{Z,t + m + 2,\epsilon} + \dbar  O_{Z,t + m + 2,\epsilon}.
    \end{multline}
    It follows that, for fixed $m$ and $t$,
    \begin{multline*}
        \sum_{k = 1}^{N-m}
        (-1)^k
        \dbar \chi_\epsilon \wedge \tr \big(
        \sigma_{k + m} \dbar \sigma_{k + m-1}\cdots \dbar \sigma_k
        (D\varphi_{k}\dbar\sigma_k)^{t + 1}
        D\varphi_k\cdots D\varphi_{k + m} \big)
        = \\
        \sum_{k = 1}^{N-m}
        (-1)^k
        \dbar \chi_\epsilon \wedge \tr \big(
        \sigma_{k + m} \dbar\sigma_{k + m-1}\cdots \dbar \sigma_{k-1}
        (D\varphi_{k-1}\dbar\sigma_{k-1})^{t}
        D\varphi_{k-1}\cdots D\varphi_{k + m}\big)
        \\ - \rho_{0,m}^{t} + (-1)^{N-m} \rho_{N-m,m}^{t} + O_{Z,t + m + 2,\epsilon} + \dbar O_{Z,t + m + 2,\epsilon}.
    \end{multline*}
    Thus, since $\rho_{0,m}^{r,s}$ and $\rho_{N-m, m}^{r,s}$ vanish,
    \begin{multline}\label{torskabotten}
        \sum_{k = 1}^{N-m} (-1)^k
        \dbar \chi_\epsilon \wedge \tr \big(
        \sigma_{k + m} \dbar \sigma_{k + m-1}\cdots \dbar \sigma_k
        (D\varphi_{k}\dbar\sigma_k)^{t + 1}
        D\varphi_k\cdots D\varphi_{k + m} \big)
        = \\
       -\sum_{k = 1}^{N-m-1} (-1)^k
        \dbar \chi_\epsilon \wedge \tr \big(
        \sigma_{k + m + 1} \dbar \sigma_{k + m}\cdots \dbar \sigma_k
        (D\varphi_{k}\dbar\sigma_k)^{t}
        D\varphi_k\cdots D\varphi_{k + m + 1} \big) \\ + O_{Z,t + m + 2, \epsilon} + \dbar  O_{Z,t + m + 2,\epsilon}.
      \end{multline}
Assume that $2\leq \ell\leq N$. By inductively applying
\eqref{torskabotten} to $m = 0, \ldots, \ell-2$ with $t = \ell-2-m$,
we get
\begin{align*}
  \sum_{k = 1}^{N}
  (-1)^k &\dbar \chi_\epsilon \wedge
  \tr \big(\sigma_k (D\varphi_k \dbar\sigma_k)^{\ell-1}
  D\varphi_k \big) \\
  &= -\sum_{k = 1}^{N-1}
    (-1)^k \dbar \chi_\epsilon \wedge \tr
    \big(\sigma_{k + 1} \dbar\sigma_k
    (D\varphi_k \dbar\sigma_k)^{\ell-2}
    D\varphi_k D\varphi_{k + 1} \big) + O_{Z,\ell,\epsilon}
    + \dbar O_{Z,\ell,\epsilon} \\
  &= \cdots \\
  &= (-1)^{\ell-1}\sum_{k = 1}^{N-\ell + 1}
    (-1)^k \dbar \chi_\epsilon \wedge \tr
    \big(\sigma_{k + \ell-1} \dbar\sigma_{k + \ell-2}
    \cdots \dbar\sigma_k D\varphi_k \cdots D\varphi_{k + \ell-1} \big)
    + O_{Z,\ell,\epsilon} + \dbar O_{Z,\ell,\epsilon},
\end{align*}
which after a shift in indices is exactly \eqref{sjunga}.  If $\ell>N$
and we perform the same induction, after $N-1$ steps we end up with
\[
  (-1)^N
  \dbar \chi_\epsilon \wedge \tr
  \big(\sigma_{N} \dbar\sigma_{N-1}\cdots \dbar\sigma_1  (D\varphi_1 \dbar\sigma_1)^{\ell-N}
  D\varphi_1\cdots D\varphi_{N} \big) + O_{Z,\ell,\epsilon} + \dbar O_{Z,\ell,\epsilon},
\]
which by \eqref{algebra} equals $\rho_{0,
  N-1}^{\ell-N-1} + \rho_{1,N-1}^{\ell-N-1} + O_{Z,\ell,\epsilon} = 
O_{Z,\ell,\epsilon}$; thus \eqref{sjunga} holds also in this case.

\smallskip

    It remains to prove \eqref{unge}. To do this
    let us replace the first factor
    $D\varphi_{k}\dbar \sigma_{k}$ in the left hand side of \eqref{unge} by
    the right hand side of \eqref{kalkon}; we then get three terms. The term
    corresponding to the second term $\dbar\sigma_{k-1}D\varphi_{k-1}$ in \eqref{kalkon} is
    precisely the first term in the right hand side of \eqref{unge}.
    Next, by \eqref{exakt} and
    \eqref{dstreck},
    \[
        \varphi_{k-1} D\varphi_k \dbar \sigma_k = 
        D \varphi_{k-1} \varphi_k \dbar \sigma_k = D\varphi_{k-1} \dbar
        \sigma_{k-1} \varphi_{k-1}.
    \]
    Applying this repeatedly we get
    \[
        \varphi_{k-1} (D\varphi_k \dbar \sigma_k)^s = (D\varphi_{k-1} \dbar
        \sigma_{k-1})^s \varphi_{k-1}.
    \]
    Using this and \eqref{exakt} (to ``move'' the $D$), we get
    \begin{multline*}
        \sigma_{k + m} \dbar \sigma_{k + m-1}\cdots \dbar\sigma_k
        (\dbar \sigma_{k-1} D\varphi_{k-1})^r
        D\dbar \sigma_{k-1} \varphi_{k-1}
        (D\varphi_{k}\dbar\sigma_k)^{s}
        D\varphi_k\cdots D\varphi_{k + m}
       = \\
        \sigma_{k + m} \dbar \sigma_{k + m-1}\cdots \dbar\sigma_k
        (\dbar \sigma_{k-1} D\varphi_{k-1})^r
        D\dbar \sigma_{k-1}
        (D\varphi_{k-1}\dbar\sigma_{k-1})^{s}
        D \varphi_{k-1} \cdots D\varphi_{k + m-1} \varphi_{k + m}.
    \end{multline*}
    It follows that the term corresponding to the first term in \eqref{kalkon}
    equals $\rho_{k-1,m}^{r,s}$.

    Finally we consider the term corresponding to the last term in
    \eqref{kalkon}.
    As above, using \eqref{exakt} and \eqref{dstreck}, we get that
    \[
        (\dbar \sigma_{k-1} D\varphi_{k-1})^r \varphi_k = 
        \varphi_k (\dbar \sigma_{k} D\varphi_{k})^r
    \]
    and thus, using this and \eqref{dexakt} (to ``move'' the $\dbar$),
    \begin{multline}\label{tulpan}
        \sigma_{k + m} \dbar \sigma_{k + m-1}\cdots \dbar\sigma_k
        (\dbar \sigma_{k-1} D\varphi_{k-1})^r
        \varphi_k D\dbar \sigma_{k}
        (D\varphi_{k}\dbar\sigma_k)^{s}
        D\varphi_k\cdots D\varphi_{k + m}
        = \\
        \dbar \sigma_{k + m} \cdots \dbar\sigma_{k + 1} \sigma_k
        \varphi_k(\dbar \sigma_{k} D\varphi_{k})^r
        D\dbar \sigma_{k}
        (D\varphi_{k}\dbar\sigma_k)^{s}
        D\varphi_k\cdots D\varphi_{k + m}.
    \end{multline}
    In view of \eqref{identitet} we can replace the factor $\sigma_k \varphi_k$ by
    $\Id_{E_k}-\varphi_{k + 1}\sigma_{k + 1}$:
    \begin{multline} \label{tulpan2}
        \dbar \sigma_{k + m} \cdots \dbar\sigma_{k + 1} \sigma_k
        \varphi_k(\dbar \sigma_{k} D\varphi_{k})^r
        D\dbar \sigma_{k}
        (D\varphi_{k}\dbar\sigma_k)^{s}
        D\varphi_k\cdots D\varphi_{k + m} = \\
        \dbar \sigma_{k + m} \cdots \dbar\sigma_{k + 1} (\dbar \sigma_{k} D\varphi_{k})^r
        D\dbar \sigma_{k}
        (D\varphi_{k}\dbar\sigma_k)^{s}
        D\varphi_k\cdots D\varphi_{k + m} + \\
        (-\dbar \sigma_{k + m} \cdots \dbar\sigma_{k + 1} \varphi_{k + 1}\sigma_{k + 1}
        (\dbar \sigma_{k} D\varphi_{k})^r
        D\dbar \sigma_{k}
        (D\varphi_{k}\dbar\sigma_k)^{s}
        D\varphi_k\cdots D\varphi_{k + m}) = : \xi + \delta.
    \end{multline}
  By repeatedly using \eqref{dstreck} we get that
    \[
        \dbar \sigma_{k + m} \cdots \dbar\sigma_{k + 1}
        \varphi_{k + 1}\sigma_{k + 1} = 
        \dbar \sigma_{k + m} \cdots \dbar\sigma_{k + 2}
        \varphi_{k + 2}\dbar \sigma_{k + 2}\sigma_{k + 1} = \dots = 
        \varphi_{k + m + 1} \dbar \sigma_{k + m + 1} \cdots \dbar\sigma_{k + 2}
        \sigma_{k + 1}.
    \]
    It follows, using \eqref{dexakt}, that $\delta$ in \eqref{tulpan2} equals
    \[
      \delta = - \varphi_{k + m + 1} \sigma_{k + m + 1} \dbar \sigma_{k + m}
      \cdots \dbar\sigma_{k + 1} (\dbar \sigma_{k} D\varphi_{k})^r
        D\dbar \sigma_{k}
        (D\varphi_{k}\dbar \sigma_{k})^{s}
        D\varphi_k\cdots D\varphi_{k + m} = : -\varphi_{k + m + 1}\beta.
    \]
    Note that $\deg \beta = 4\kappa + 2$ and $\deg_e \beta = 2\kappa + 1$, where
    $\kappa = m + r + s + 1$, and since  
\[
  \deg \varphi_{k + m + 1} = \deg_e \varphi_{k + m + 1} = 1,
\]
    we get, in view of \eqref{spar}, that
    $\tr (\varphi_{k + m + 1}\beta) = - \tr (\beta\varphi_{k + m + 1})$ and
    it follows that this terms equals
   $\rho_{k,m}^{r,s}$.

It remains to consider $\xi$ in \eqref{tulpan2}.
Let
\begin{equation*}
        \eta:= \sigma_{k + m} \dbar \sigma_{k + m-1}\cdots \dbar\sigma_{k + 1} (\dbar \sigma_{k} D\varphi_{k})^r
        D\dbar \sigma_{k}
        (D\varphi_{k}\dbar \sigma_{k})^{s}
        D\varphi_k\cdots D\varphi_{k + m}
      \end{equation*}
Then, by \eqref{eq:dbarLeibniz},
$\dbar \eta = \xi + \xi'$, where $\xi'$ consists of a sum of terms with a factor
$\dbar D \varphi_j$ or $\dbar D\dbar \sigma_j$.
Let $q = m + r + s + 2$. Then, note that $\dbar \chi_\epsilon  \wedge \eta$ is in
$O_{Z,q,\epsilon}$, and thus
$\dbar \chi_\epsilon \wedge \dbar\eta = -\dbar (\dbar \chi_\epsilon
\wedge \eta) \in \dbar O_{Z,q,\epsilon}$.
Moreover, by \eqref{eq:dbarD}, each term in $\xi'$ has a factor that is a smooth
$(1,1)$-form. Therefore $\dbar \chi_\epsilon \wedge \xi'\in
O_{Z,q,\epsilon}$, and hence $\tr \dbar\chi_\epsilon\wedge \xi = - \tr(\dbar\chi_\epsilon\wedge \xi') + \tr(\dbar\chi_\epsilon \wedge \dbar\eta)  \in O_{Z, q,
  \epsilon} + \dbar O_{Z,q,\epsilon}$. This concludes the proof of \eqref{unge}.
\end{proof}

\begin{proof}[Proof of Theorem ~\ref{saltkar}]
We first prove \eqref{glashus2}. Since $Z$ has codimension $p$ and
$\chi^p \sim \chi_{[1,\infty)}$, by Lemmas~\ref{lma:OdimensionPrinciple}, \ref{struntsak}, and~\ref{klunga},
and by~\eqref{eq:Rdef}, we have
\begin{align*}
  \ch_p^{\Res}(E,D)
  & = \lim_{\epsilon \to 0}  \ch_p(E,\widehat D^\epsilon) \\
  & = \frac{1}{(2\pi i)^p p!}  \lim_{\epsilon \to
    0} \dbar\chi_\epsilon^p \wedge \sum_{k = 1}^N (-1)^k
    \tr \big(\sigma_k D\varphi_k(\dbar\sigma_k
    D\varphi_k)^{p-1}\big) \\
  & = \frac{(-1)^p}{(2\pi i)^p p!} \sum_{k = 0}^{N-p}
    (-1)^k \lim_{\epsilon \to 0} \dbar \chi_\epsilon^p \wedge \tr
    \big(\sigma_{k + p} \dbar\sigma_{k + p-1} \cdots \dbar\sigma_{k + 1}
    D\varphi_{k + 1} \cdots D\varphi_{k + p} \big) \\
  & = \frac{(-1)^p}{(2\pi i)^p p!} \sum_{k = 0}^{N-p}
    (-1)^k \tr\big(
    R^k_{k + p} D\varphi_{k + 1} \cdots D\varphi_{k + p}
    \big).
\end{align*}
Since $D\varphi_{k + 1}\cdots D\varphi_{k + p}$ and $R^k_{k + p}$ both have
total degree $2p$ 
and endomorphism degree $p$, it follows from \eqref{spar} that
\begin{equation*}
  \tr \big(R^{k}_{k + p}D\varphi_{k + 1} \cdots D\varphi_{k + p} \big)
  = (-1)^{p} \tr \big(D\varphi_{k + 1} \cdots D\varphi_{k + p} R^{k}_{k + p}
  \big),
\end{equation*}
and thus \eqref{glashus2} follows.

Next, by Theorem \ref{thm:existence2} and Remark \ref{locket},
$\ch_\ell^{\Res}(E, D)$ is a pseudomeromorphic current with support on $Z$
and with components of bidegree $(\ell + q, \ell-q)$ where $q\geq
0$. Therefore it vanishes by the dimension principle when $\ell<p$,
see Section \ref{hemma}. This proves \eqref{glashus3}.

It remains to prove \eqref{glashus4}. From the beginning of the proof
of Lemma~\ref{struntsak} and \eqref{eq:Oideal} it follows that
\begin{equation}
  \label{attack}
  \ch_{\ell_1}(E,\widehat{D}^\epsilon) \wedge
  \dots \wedge \ch_{\ell_m}(E,\widehat{D}^\epsilon)
 = C {p}_{(\ell_1, \ell_1)}(E, \widehat D^\epsilon)
  \wedge \cdots \wedge
  {p}_{(\ell_m, \ell_m)}(E, \widehat D^\epsilon) + O_{Z,\ell_1 + \cdots
    + \ell_m,\epsilon},
\end{equation}
for some appropriate constant $C$.
By \eqref{eq:pl}, the fact that $(\dbar \chi_\epsilon)^2 = 0$, and
\eqref{eq:Oideal}, it follows that
\[
  {p}_{(\ell_1, \ell_1)}(E, \widehat D^\epsilon) \wedge \cdots \wedge
    {p}_{(\ell_m, \ell_m)}(E, \widehat D^\epsilon)\in
    O_{Z,p,\epsilon} 
\]
if $m \geq 2$ and $\ell_1 + \cdots + \ell_m\leq p$.  Thus the limit of
\eqref{attack} vanishes in this case, which proves \eqref{glashus4}.
\end{proof}

Assume that $(E,\varphi)$ is a complex of Hermitian vector bundles of the form
    \eqref{turkos} such that $\mathcal{H}_k(E)$ has pure codimension
    $p$ or vanishes for $k = 0,\dots,N$, and let $Z = \cup \supp
    \mathcal{H}_k(E)$. Then
        $(E,\varphi)$ is pointwise exact outside $Z$, which has codimension
        $p$. Now, by combining Theorem~\ref{thm:LW3} and Theorem~\ref{thm:mainGeneral},
we obtain the following generalization of Theorem~\ref{cykla}.
\begin{corollary} \label{lydia}
    Assume that $(E,\varphi)$ is a complex of Hermitian vector bundles of the form
    \eqref{turkos} such that $\mathcal{H}_k(E)$ has pure codimension
    $p$ or vanishes for $k = 0,\dots,N$. Moreover, assume that $D = (D_k)$ is a
    $(1,0)$-connection on $(E, \varphi)$.
    Then
    \begin{equation*}
        c_p^{\Res}(E, D) = (-1)^{p-1}(p-1)![E].
      \end{equation*}
      Moreover \eqref{hajka} and \eqref{bajka}
      hold.
  \end{corollary}
  Here $[E]$ is the cycle of $(E, \varphi)$ defined by \eqref{cycleofcomplex}.
In particular, it follows from Corollary \ref{lydia} that $c_p^{\Res}(E, D)$ is independent of the
choice of Hermitian metric and $(1,0)$-connection $D$ on $(E,
\varphi)$.

By equipping a complex of vector bundles $(E, \varphi)$ with Hermitian
metrics and $(1,0)$-connections and taking cohomology we get the
following generalization of \eqref{janakippo}.
\begin{corollary} \label{cor:chernFundCycleGeneral}
  Assume that $(E,\varphi)$ is a complex of vector bundles of the form
\eqref{turkos} such that $\mathcal{H}_k(E)$ has pure codimension $p$
or vanishes for $k = 0,\dots,N$.  Then
\[
  \big[(-1)^{p-1}(p-1)![E]\big] = c_p(E).
\]
\end{corollary}

\section{An example}\label{triumf}

We will compute (products of) Chern currents
$c^{\Res}(E,D)$ for an explicit choice of $(E, \varphi)$ and $D$.
Let $\mathcal{J} \subseteq \Ok_{\PP^2_{[t,x,y]}}$ be defined by
$\mathcal{J} = \mathcal{J}(y^k,x^\ell y^m)$, where $m < k$,
and let $\Ok_Z := \Ok_{\PP^2}/\mathcal{J}$.
Then $Z$ has pure dimension $1$, since $Z_\text{red} = \{y = 0\}$,
which is irreducible. However, note that $\mathcal J$
has an embedded prime $\mathcal J(x,y)$ of dimension
$0$.
Now $\mathcal F = \Ok_Z$ has a locally free resolution of the
form
\begin{equation}\label{pellefant}
    0 \to \Ok(-k-\ell) \xrightarrow[]{\varphi_2}
    \Ok(-k)\oplus\Ok(-\ell-m) \xrightarrow[]{\varphi_1} \Ok_{\PP^2}\to
    \mathcal F\to 0,
\end{equation}
where in the trivialization in the coordinate chart $\C^2 = 
\C^2_{(x,y)}$
\begin{equation}
  \label{elefant}
  \varphi_2 = \left[ \begin{array}{c} -x^\ell \\ y^{k-m} \end{array} \right] \text { and }
        \varphi_1 = \left[ \begin{array}{cc} y^k & x^\ell
            y^m \end{array} \right].
\end{equation}

Let us start by computing the Chern class of $\mathcal F$.
Let $\omega$ denote $c_1(\Ok(1))$.
Then
\begin{align*}
  c(E_0) & = c(\Ok_{\PP^2}) = 1 \\
  c(E_1) & = c\big(\Ok(-k)\oplus\Ok(-\ell-m) \big) = 
  1 - (k + \ell + m)\omega + k(\ell + m)\omega^2\\
  c(E_2) & = c\big(\Ok(-k-\ell) \big) = 1-(k + \ell)\omega,
\end{align*}
  see \emph{e.g.} \cite[Chapter~3]{Fulton}.
Moreover,
\begin{equation}
  c(E_1)^{-1} = 1-c_1(E_1) + c_1(E_1)^2-c_2(E_1);
\end{equation}
here the sum ends in degree $2$, since we are in dimension $2$.
Thus, by \eqref{eq:cClassSheafDef},
\begin{align*}
  c(\mathcal F) = c(E_0)c(E_1)^{-1}c(E_2)
  & = 1 - c_1(E_1) + c_1(E_2) + c_1(E_1)^2 - c_2(E_1) -
  c_1(E_1)c_1(E_2) \\
  & = 1 + m\omega + \big(m^2 + \ell (m-k)\big) \omega^2.
\end{align*}
In particular,
\begin{equation}\label{faren}
  c_1(\mathcal F) = m\omega
  \quad \text{ and } \quad
  c_2(\mathcal F) = \big(m^2 + \ell (m-k)\big) \omega^2.
\end{equation}

\subsection{Chern currents}\label{fron}

Assume that each $E_k$ in \eqref{pellefant}
is equipped
with the metric induced by the standard metric on
$\mathcal O(1)\to \PP^2$, in turn induced by the standard metric on
$\C^{3}$,
and let $D_k$ the corresponding Chern connection.
Let $D = (D_k)$, let $\chi\sim\chi_{[1,\infty)}$, and let $\chi_\epsilon$
and $\widehat D^\epsilon$ be as in Section \ref{taget}.
By Theorem \ref{cykla},
\begin{equation*}
  c_1^{\Res}(E, D) = [\mathcal F] = [Z] = m[y = 0],
\end{equation*}
which clearly is a representative of $c_1(\mathcal F)$, see
\eqref{faren}, with support on $\supp \mathcal F = Z_{\text{red}} = \{y = 0\}$.
We want to compute $c^{\Res}_2(E, D)$ and $(c^{\Res}_1)^2(E,D)$.
Note that these currents are not covered by Theorem \ref{thm:main}, since
$p = 1$ in this case.

Let $\hat p_\ell = p_{(\ell,\ell)}(E,\widehat D^\epsilon)$, where
$p_{(\ell,\ell)}$ is given by \eqref{eq:pl11}.
For degree reasons, $\ch_2(E,\widehat D^\epsilon) = \ch_{(2,2)}(E,\widehat D^\epsilon)$ and
$\ch_1^2(E,\widehat D^\epsilon) = \ch_{(1,1)}^2(E,\widehat D^\epsilon)$,
\emph{cf.}\ Remark~\ref{locket}.
It follows in view of \eqref{eq:ch12} and
\eqref{eq:chernCharacterAlternativell} that
\begin{equation} \label{begynnelse}
    c_2(E,\widehat D^\epsilon) = 
    \left(\frac{i}{2\pi}\right)^2\frac{1}{2} (\hat p_1^2-\hat p_2)
    \quad \text{ and } \quad
        c_1^2(E,\widehat D^\epsilon) = \left(\frac{i}{2\pi}\right)^2 \hat p_1^2.
\end{equation}
Thus, to compute $c_2^{\Res}(E,D)$ and $(c_1^{\Res})^2(E,D)$, it suffices to calculate the limits of $\hat p_1^2$ and $\hat p_2$ as $\epsilon \to 0$.

\smallskip

Note first that only two of the standard coordinate charts of $\PP^2$ intersect $Z$. In $\C^2_{(t,y)}$,
we have that $\varphi_1 = y^m \left[ \begin{array}{cc} y^{k-m} &  1 \end{array} \right]$,
so $\sigma_1 = (1/y^m)\sigma_1'$, where $\sigma_1'$ is smooth.
By using \eqref{trubbel} one can check that the limits of $\hat p_1^2$ and $\hat p_2$
put no mass at $\{t = y = 0\}$. Thus
it is enough to compute the limits in the coordinate chart $\C^2_{(x,y)}$
where $\varphi_j$ are given by \eqref{elefant}.
Note that $\varphi_1 = y^m \varphi_1'$, where $\varphi_1' = [y^{k-m} ~~~
x^\ell]$ has rank $1$ outside of the origin.
Then $\sigma_1 = (1/y^m)\sigma_1'$, where $\sigma_1'$ is smooth outside the origin, and
\begin{equation} \label{eq:sigmaphiprim}
  \sigma_1'\varphi_1'|_{\{y = 0\}} = \left[\begin{array}{cc} 0 & 0 \\ 0
                                                             &
                                                               1 \end{array}\right]
\end{equation}
when $x\neq 0$ and $\varphi_1'\sigma_1' = 1$ outside the origin. Also note that $\sigma_2$ is smooth outside the origin,
since $\varphi_2$ has constant rank there.
Let $O'_{Z,\ell,\epsilon}$ be defined as in the
beginning of Section~\ref{chernsec} but with $\sigma_1$ replaced by
$\sigma_1'$, and let $O_\epsilon = O'_{Z,2,\epsilon}$. Then
$\psi_\epsilon\in O_\epsilon$ is smooth outside the origin and, by arguments as in the proof of Lemma~\ref{lma:OdimensionPrinciple}, $\lim_{\epsilon\to 0}\psi_\epsilon = 0$.

Next, let $\hat\omega = (2\pi /i) \omega$, where $\omega$ now denotes the
Fubini-Study form. Then
\begin{equation}\label{skriket}
  \Theta_1 = \left[
    \begin{array}{cc} -k
      & 0 \\ 0 & -(\ell + m)
    \end{array}
  \right]\hat\omega,
  \quad \Theta_2 = -(k + \ell)\hat\omega.
\end{equation}
In particular, $\Theta_k$ is of bidegree $(1,1)$.  Let 
$\widetilde D = (\widetilde D_k)$ be the connection on 
$\PP^2\setminus Z$ defined by \eqref{bibliotek} and let 
$\widetilde \Theta_k$ be the corresponding curvature forms. Then a
computation, \emph{cf.}\ \eqref{snabel} and \eqref{fabel}, yields
\begin{equation}\label{wasta}
(\widehat\Theta^\epsilon_k)_{(1,1)} = 
-\dbar\chi_\epsilon \wedge \sigma_k D\varphi_k + 
\chi_\epsilon  (\widetilde\Theta_k)_{(1,1)} + 
(1-\chi_\epsilon)\Theta_k.
\end{equation}

\smallskip

Let us start by computing $\hat p_1^2$.  Recall from the proof of
Lemma \ref{struntsak} that 
$\tilde p_j = - \tr (\widetilde\Theta_1)_{(1,1)}^j + \tr
(\widetilde\Theta_2)_{(1,1)}^j $
vanishes where $\chi_\epsilon\not\equiv 0$ for $j = 1,2$. Moreover, note
in view of \eqref{skriket} that 
$-\tr \Theta_1 + \tr \Theta_2 = m\hat\omega$.  It follows that
\begin{equation*}
\hat p_1 = 
 -\tr (\widehat\Theta^\epsilon_1)_{(1,1)} + \tr (\widehat\Theta^\epsilon_2)_{(1,1)}
 = \\
\dbar\chi_\epsilon \wedge \big(\tr (\sigma_1 D\varphi_1) -
\tr (\sigma_2 D\varphi_2)\big) + (1-\chi_\epsilon)m\hat\omega.
\end{equation*}
Note that $(1-\chi_\epsilon)\dbar\chi_\epsilon = (1/2)\dbar\tilde{\chi}_\epsilon$, where
$\tilde{\chi} = 2(\chi-\chi^2/2) \sim \chi_{[1,\infty)}$.
Using this and that $(\dbar \chi_\epsilon)^2 = 0$, we
get
\begin{align*}
  \hat p_1^2 &= 
  2m \hat\omega \wedge (1-\chi_\epsilon)\dbar \chi_\epsilon \wedge \big(\tr (\sigma_1 D\varphi_1) -
\tr (\sigma_2 D\varphi_2)\big) + (1-\chi_\epsilon)^2m^2\hat\omega^2\\
 &= m\hat\omega \wedge \dbar\tilde{\chi}_\epsilon \wedge  \tr (\sigma_1
  D\varphi_1) + O_\epsilon.
\end{align*}
Note that
\begin{equation}\label{bokhylla}
\sigma_1 D\varphi_1 = -\frac{Dy^m}{y^m}\sigma_1'\varphi_1' + \theta_1
\sigma_1'\varphi_1' + \sigma_1'D\varphi_1'.
\end{equation}
Therefore, in view of the Poincar\'e-Lelong formula, \emph{cf.}\
\eqref{oklart} and \eqref{eq:sigmaphiprim}, since $\sigma_1'\varphi_1'$ is smooth outside the origin,
\begin{equation}\label{krispig}
  \dbar\tilde{\chi}_\epsilon \wedge \sigma_1 D\varphi_1 = 
-\dbar\tilde{\chi}_\epsilon \wedge   \frac{Dy^m}{y^m} \sigma_1'\varphi_1'  + 
O_\epsilon
\xrightarrow[\epsilon\to 0]{}
\frac{2\pi}{i} m [y = 0]\left[ \begin{array}{cc} 0 & 0 \\ 0 & 1 \end{array} \right]
\end{equation}
outside the origin. Since the limit is a pseudomeromorphic $(1,1)$-current, \eqref{krispig} holds everywhere by
the dimension principle.
It follows that
\begin{equation}\label{huskvarna}
  \lim_{\epsilon\to 0} \hat p_1^2 = 
  \lim_{\epsilon\to 0}  m\hat\omega \wedge \dbar\tilde{\chi}_\epsilon \wedge  \tr (\sigma_1
  D\varphi_1)
  = \frac{2\pi}{i}m^2 \hat\omega \wedge [y = 0].
 \end{equation}

\smallskip

Let us next consider $\hat p_2 = -\tr (\widehat \Theta_1)^2_{(1,1)} + \tr
(\widehat \Theta_2)_{(1,1)}^2$.
A computation using \eqref{spar}, \emph{cf.}\ \eqref{wasta},
yields
\begin{align}\label{prinsessan}
    \tr (\widehat\Theta^\epsilon_k)_{(1,1)}^2 &= 
\tr \big(-\dbar\chi_\epsilon \wedge \sigma_k D\varphi_k - \chi_\epsilon \dbar(\sigma_k
D\varphi_k) + \Theta_k\big)^2\\
&=\dbar\chi_\epsilon^2 \wedge \tr \big(\sigma_k D\varphi_k\dbar(\sigma_k
D\varphi_k)\big)-2\dbar\chi_\epsilon \wedge \tr (\sigma_k D\varphi_k \Theta_k)
 + \tr \big(\chi_\epsilon (\widetilde\Theta_k)_{(1,1)} + (1-\chi_\epsilon)
\Theta_k\big)^2.\nonumber
\end{align}
Again using that $\tilde p_j = -\tr (\widetilde \Theta_1)_{(1,1)}^j + \tr (\widetilde
\Theta_2)_{(1,1)}^j$ vanishes where $\chi_\epsilon\not\equiv 0$,
we get
\begin{equation}\label{aina}
- \big(\chi_\epsilon (\widetilde\Theta_1)_{(1,1)} + (1-\chi_\epsilon)
\Theta_1\big)^2 + 
\big(\chi_\epsilon (\widetilde\Theta_2)_{(1,1)} + (1-\chi_\epsilon)
\Theta_2\big)^2 = O_\epsilon \to 0.
\end{equation}
Note that $\dbar\chi_\epsilon \wedge\sigma_2 D\varphi_2
\Theta_2$ is in $O_\epsilon$. Therefore, in view of \eqref{skriket}
and \eqref{krispig},
\begin{multline}\label{hemlig}
2\dbar\chi_\epsilon \wedge \tr (\sigma_1 D\varphi_1 \Theta_1) -2\dbar\chi_\epsilon \wedge\tr (\sigma_2 D\varphi_2 \Theta_2)
 = \\
-2\dbar\chi_\epsilon \wedge
\frac{Dy^m}{y^m}\tr (\sigma_1' \varphi_1' \Theta_1) + O_\epsilon
\xrightarrow[\epsilon\to 0]{}
- \frac{2\pi}{i}2m (\ell + m)\hat\omega \wedge [y = 0].
\end{multline}

Let us next consider the contribution from the first term
\begin{equation}\label{hundratals}
\dbar\chi_\epsilon^2 \wedge \tr \big(\sigma_k D\varphi_k\dbar(\sigma_k
D\varphi_k)\big)
 = 
\dbar\chi_\epsilon^2 \wedge \tr \big(\sigma_k D\varphi_k\dbar\sigma_k
D\varphi_k\big)
-
\dbar\chi_\epsilon^2 \wedge \tr \big(\sigma_k D\varphi_k\sigma_k
\dbar(D\varphi_k)\big)
\end{equation}
in \eqref{prinsessan}. We start by considering the contribution from the first term in
\eqref{hundratals}. By arguments as in the proof of Lemma~\ref{klunga}, we get that
\begin{multline}\label{ekeskog}
 - \dbar\chi_\epsilon^2 \wedge \tr \big(\sigma_1 D\varphi_1 \dbar \sigma_1 D\varphi_1 \big) + 
  \dbar\chi_\epsilon^2 \wedge \tr \big(\sigma_2 D\varphi_2 \dbar \sigma_2
  D\varphi_2 \big)
  = \\
  \dbar\chi_\epsilon^2 \wedge \tr \big(\sigma_2 \dbar \sigma_1 D\varphi_1  D\varphi_2 \big)
 - \dbar\chi_\epsilon^2 \wedge \tr \big( D \dbar \sigma_1 D\varphi_1 \big) + 
 \dbar\chi_\epsilon^2 \wedge \tr \big( D \dbar \sigma_2 D\varphi_2 \big).
\end{multline}
Taking the limit of the first term in the right hand side of
\eqref{ekeskog}, we get
\begin{equation}
  \label{tiotals}
  \begin{split}
    \lim_{\epsilon\to 0} \dbar\chi_\epsilon^2 \wedge
    \tr (\sigma_2 \dbar \sigma_1 D\varphi_1 D\varphi_2)
    &= \lim_{\epsilon\to 0} \dbar\chi_\epsilon^2 \wedge
    \tr (D\varphi_1 D\varphi_2\sigma_2 \dbar \sigma_1) \\
    &= \tr (D\varphi_1 D\varphi_2 R_2^0) \\\
    &= \left (\frac{2\pi}{i} \right)^2 \ell (2k-m)[0].
  \end{split}
\end{equation}
Here, the first equality follows from
\eqref{spar}, the second equality from \eqref{eq:Rdef},
and the third equality is computed in
\cite[Example 5.2]{LW2}.
Next, by
\eqref{spar},
$\tr ( D \dbar \sigma_1 D\varphi_1) = 
-\tr ( D\varphi_1  D \dbar \sigma_1 )$.
Using \eqref{bord}, \eqref{karlsborg} and the fact that
$\varphi_1'\sigma_1' = 1$, so $\varphi_1' \dbar\sigma_1' = 0$, we have
\[
 \dbar \chi_\epsilon^2 \wedge  D\varphi_1  D \dbar \sigma_1  = 
 \dbar \chi_\epsilon^2\wedge  D(D\varphi_1  \dbar \sigma_1)
 = 
 \dbar \chi_\epsilon^2\wedge D(D\varphi'_1  \dbar \sigma'_1)
 = \dbar \chi_\epsilon^2 \wedge D\varphi'_1 D\dbar\sigma'_1.
\]
If we let $f$ be the section of $\Ok(k-m)\oplus\Ok(\ell)$ defined by $f = \left[\begin{array}{cc} y^{k-m} & x^\ell \end{array}\right]$,
then $\varphi_2,\varphi_1'$ are the morphisms in the Koszul complex defined by (contraction with) $f$.
If we let $\sigma$ be the minimal inverse of $f$, when $f$ is viewed as a section of $\Hom(\Ok(-(k-m))\oplus\Ok(-\ell),\Ok)$,
then $\sigma_2$ and $\sigma_1'$ are given by multiplication with $\sigma$. One may verify that $D\varphi_2$ and $D\varphi_1'$
are given by contraction with $Df$, and that $D\dbar\sigma_2$ and $D\dbar\sigma_1'$ are given by multiplication with $D\dbar\sigma$.
A calculation then yields that
\[
	\tr \big(D\varphi'_1 D\dbar\sigma'_1 \big) = - \tr \big( D \dbar \sigma_2 D\varphi_2\big),
\]
so by \eqref{ekeskog},
\begin{equation}\label{ekeskog2}
 - \dbar\chi_\epsilon^2 \wedge \tr \big(\sigma_1 D\varphi_1 \dbar \sigma_1 D\varphi_1 \big) +
  \dbar\chi_\epsilon^2 \wedge \tr \big(\sigma_2 D\varphi_2 \dbar \sigma_2
  D\varphi_2 \big)
  = \dbar\chi_\epsilon^2 \wedge \tr \big(\sigma_2 \dbar \sigma_1 D\varphi_1  D\varphi_2 \big).
\end{equation}
Thus, in view of \eqref{ekeskog2} and \eqref{tiotals},
\begin{equation}\label{fixatill}
- \dbar\chi_\epsilon^2 \wedge \tr \big(\sigma_1 D\varphi_1 \dbar \sigma_1 D\varphi_1 \big) + 
  \dbar\chi_\epsilon^2 \wedge \tr \big(\sigma_2 D\varphi_2 \dbar \sigma_2
  D\varphi_2 \big) \xrightarrow[\epsilon\to 0]{} \left (\frac{2\pi}{i} \right)^2 \ell (2k-m)[0].
  \end{equation}

Next, let us consider the contribution from the second term in
\eqref{hundratals}.
As above, using \eqref{eq:dbarD}, \emph{cf.}\ \eqref{bokhylla},
\begin{equation*}
\dbar\chi_\epsilon^2 \wedge
\big(\sigma_1 D\varphi_1\sigma_1
\dbar (D\varphi_1) \big)
 = 
-\dbar\chi_\epsilon^2 \wedge
\big(\sigma_1 D\varphi_1 \sigma_1 \varphi_1 \Theta_1 \big)
 = 
\dbar\chi_\epsilon^2 \wedge
\big( \frac{Dy^m}{y^m}\sigma_1'\varphi_1'\Theta_1 \big) + O_\epsilon.
\end{equation*}
Note that
$\dbar\chi_\epsilon^2 \wedge \sigma_2 D\varphi_2\sigma_2
\dbar (D\varphi_2)$ is in $O_\epsilon$.
Thus, by \eqref{krispig} and \eqref{skriket},
\begin{equation}
  \label{tjog}
  \dbar\chi_\epsilon^2 \wedge \tr \big(\sigma_1 D\varphi_1 \sigma_1
  \dbar(D\varphi_1)\big)-
  \dbar\chi_\epsilon^2 \wedge \tr \big(\sigma_2 D\varphi_2 \sigma_2
  \dbar (D\varphi_2) \big)
\xrightarrow[\epsilon\to 0]{} \frac{2\pi}{i}m(\ell + m) \hat\omega \wedge [y = 0],
\end{equation}
\emph{cf.}\ \eqref{hemlig}.

From \eqref{hundratals}, \eqref{fixatill}, and \eqref{tjog},
we conclude that
\begin{multline}\label{miljontals}
 -\dbar\chi_\epsilon^2 \wedge
  \tr \big(\sigma_1 D\varphi_1 \dbar (\sigma_1 D\varphi_1) \big) + 
  \dbar\chi_\epsilon^2 \wedge \tr \big(\sigma_2 D\varphi_2 \dbar (\sigma_2
  D\varphi_2)\big)
\xrightarrow[\epsilon\to 0]{} \\
  \left (\frac{2\pi}{i}\right)^2
  \ell (2k -m) [0] + \frac{2\pi}{i}m(\ell + m) \hat\omega \wedge [y = 0].
\end{multline}

Next, from \eqref{prinsessan}, \eqref{aina}, \eqref{hemlig}, and
\eqref{miljontals}, we conclude that
\begin{equation}\label{delen}
  \hat p_2 = -\tr (\widehat \Theta_1)^2_{(1,1)} + \tr
  (\widehat \Theta_2)_{(1,1)}^2 \xrightarrow[\epsilon\to 0]{}
 -(2\pi/i)m (\ell + m) \hat\omega \wedge [y = 0] + 
  (2\pi/i)^2\ell (2k -m) [0].
\end{equation}
Finally from \eqref{begynnelse}, \eqref{huskvarna}, and \eqref{delen}
we conclude that
\[
  c_2^{\Res}(E, D)
 = \left (\frac{i}{2\pi}\right)^2 \frac{1}{2}\,
  \lim_{\epsilon\to 0} (\hat p_1^2 - \hat p_2)
 = \frac{1}{2} \big(m(2m + \ell) \omega \wedge
  [y = 0] - \ell (2k -m) [0]\big)
\]
and that
\[
  (c_1^{\Res})^2(E,D)
 = \left (\frac{i}{2\pi}\right)^2
  \lim_{\epsilon\to 0} \hat p_1^2 = m^2 \omega [y = 0].
\]
Taking cohomology, since 
$[[0]] = [[y = 0]\wedge \omega] = [\omega^2]$, we get
\begin{align*}
  \big[c_2^{\Res}(E,D)\big] = \big(m^2 + \ell(m-k)\big) [\omega^2] = 
  c_2(\mathcal F)
  \quad\text{ and }\quad
  \big[(c_1^{\Res})^2(E,D)\big] = m^2[\omega^2] = c_1(\mathcal{F})^2,
\end{align*}
see \eqref{faren}.



\begin{thebibliography}{XXXXXX+}

\bibitem[AW07]{AW1}
  M.~Andersson and E.~Wulcan,
  \emph{Residue currents with prescribed annihilator ideals},
  Ann. Sci. \'{E}cole Norm. Sup.
  \textbf{40} (2007),
  985--1007.

\bibitem[AW10]{AW2}
\bysame,
  \emph{Decomposition of residue currents},
  J. Reine Angew. Math.
  \textbf{638} (2010),
  103--118.

\bibitem[AW18]{AW3}
  \bysame,
  \emph{Direct images of semi-meromorphic currents},
  Ann. Inst. Fourier
  \textbf{68} (2018),
  875--900.

\bibitem[BB72]{BB}
  P.~Baum and R.~Bott,
  \emph{Singularities of holomorphic foliations},
  J. Differential Geometry
  \textbf{7} (1972),
  279--342.

\bibitem[BGS90a]{BGS4}
  J.-M.~Bismut, H.~Gillet, and C.~Soul\'{e},
  \emph{Bott-Chern currents and complex immersions},
  Duke Math. J. \textbf{60} (1990),
  255--284.

\bibitem[BGS90b]{BGS5}
  \bysame,
  \emph{Complex immersions and Arakelov geometry}.
  In: The Grothendieck Festschrift, Vol. I, pp.~249--331,
  Progr. Math., 86,
  Birkh\"{a}user Boston, Boston, MA, 1990.

\bibitem[BSW21]{BSW}
  J.-M.~Bismut, S.~Shen, and Z.~Wei,
  \emph{Coherent sheaves, superconnections, and RRG},
  preprint \arXiv{2102.08129} (2021).

\bibitem[BS58]{BS}
  A.~Borel and J.-P.~Serre,
  \emph{Le th\'{e}or\`eme de Riemann-Roch},
  Bull. Soc. Math. France
  \textbf{86} (1958),
  97--136.

\bibitem[CH78]{CH}
  N.~R.~Coleff and M.~E.~Herrera,
  \emph{Les courants r\'{e}siduels associ\'{e}s \`a une forme
    m\'{e}romorphe}, 
  Lecture Notes in Mathematics 633,
  Springer, Berlin, 1978.

\bibitem[EH16]{EH}
  D.~Eisenbud and J.~Harris,
  \emph{3264 and all that---a second course in algebraic geometry},
  Cambridge University Press, Cambridge, 2016.

\bibitem[Ful98]{Fulton}
  W.~Fulton,
  \emph{Intersection theory},
  Ergebnisse der Mathematik und ihrer Grenzgebiete. 3. Folge. 2,
  Springer-Verlag, Berlin, 1998.

\bibitem[Gre80]{Green}
  H.~I.~Green,
  \emph{Chern classes for coherent sheaves},
  Ph.D. Thesis,
  University of Warwick, (1980). Available from \url{http://wrap.warwick.ac.uk/40592/1/WRAP_THESIS_Green_1980.pdf}.

\bibitem[Gri10]{Gri}
  J.~Grivaux,
  \emph{Chern classes in Deligne cohomology for coherent analytic sheaves},
  Math. Ann.
  \textbf{347} (2010),
  249--284.

\bibitem[Hos20a]{Hos}
  T.~Hosgood,
  \emph{Simplicial Chern-Weil theory for coherent analytic sheaves, part I},
  preprint \arXiv{2003.10023} (2020).

\bibitem[Hos20b]{Hos2}
  \bysame,
  \emph{Simplicial Chern-Weil theory for coherent analytic sheaves, part II},
  preprint \arXiv{2003.10591} (2020).


\bibitem[HL93]{HL}
  F.~R.~Harvey and H.~B.~Lawson Jr.,
  \emph{A theory of characteristic currents associated with a singular
   connection},
  Ast\'{e}risque \textbf{213}, Soci\'et\'e Math\'ematique de France, 1993.

\bibitem[HL71]{HeLi}
  M.~Herrera and D.~Lieberman,
  \emph{Residues and principal values on complex spaces},
  Math. Ann. \textbf{194} (1971),
  259--294.

\bibitem[JL21]{JL}
  J.~Johansson and R.~L\"{a}rk\"{a}ng,
  \emph{An explicit isomorphism of different representations of the
    Ext functor using residue currents},
  preprint \arXiv{2109.00480} (2021).


\bibitem[Kar08]{Karoubi}
  M.~Karoubi,
  \emph{$K$-theory},
  Classics in Mathematics, Springer-Verlag, Berlin, 2008.

\bibitem[LW18]{LW2}
  R.~L\"{a}rk\"{a}ng, and E.~Wulcan,
  \emph{Residue currents and fundamental cycles},
  Indiana Univ. Math. J.
  \textbf{67} (2018),
  1085--1114.

\bibitem[LW21]{LW3}
  \bysame,
  \emph{Residue currents and cycles of complexes of vector bundles}, Ann. Fac. Sci. Toulouse~\textbf{30} (2021), no.~5, 961--984.
 

\bibitem[Mac95]{MacDonald}
  I.~G.~Macdonald,
  \emph{Symmetric functions and Hall polynomials},
  Oxford Mathematical Monographs,
  The Clarendon Press, Oxford University Press, New York, 1995.

\bibitem[Qia16]{Qiang}
  H.~Qiang,
  \emph{On the Bott-Chern characteristic classes for coherent sheaves},
  preprint \arXiv{1611.04238} (2016).

\bibitem[Suw98]{Su}
  T.~Suwa,
  \emph{Indices of vector fields and residues of singular holomorphic
    foliations},
  Actualit\'{e}s Math\'{e}matiques,
  Hermann, Paris, 1998.

\bibitem[TT86]{TTGreen}
  D. Toledo and Y.~L.~L.~Tong,
  \emph{Green's theory of Chern classes and the Riemann-Roch formula}.
  In: \emph{The Lefschetz centennial conference, Part I} (Mexico City, 1984),
  pp.~261--275,
  Contemp. Math. 58, Amer. Math. Soc., Providence, RI, 1986.

\bibitem[Tu17]{Tu}
  L.~W.~Tu,
  \emph{Differential geometry},
  Graduate Texts in Mathematics, 275, Springer, Cham, 2017.

\bibitem[Wu20]{Wu}
  X.~Wu,
  \emph{Intersection theory and Chern classes in Bott-Chern cohomology},
  preprint \arXiv{2011.13759} (2020).

\end{thebibliography}
\end{document}